\documentclass[10pt,reqno]{amsart}

\usepackage{a4wide}
\usepackage{dsfont} % used in the command \I
\usepackage{amssymb}
\usepackage{pstricks}

\newtheorem{proposition}{Proposition}[section]
\newtheorem{theorem}[proposition]{Theorem}
\newtheorem{lemma}[proposition]{Lemma}

\theoremstyle{remark}
\newtheorem{definition}[proposition]{Definition}
\newtheorem{remark}[proposition]{Remark}
\newtheorem{example}[proposition]{Example}

\newcommand{\cst}{\ifmmode\mathrm{C}^*\else{$\mathrm{C}^*$}\fi}
\newcommand{\wst}{\ifmmode\mathrm{C}^*\else{$\mathrm{W}^*$}\fi}
\newcommand{\st}{\;\vline\;}
\newcommand{\CC}{\mathbb{C}}
\newcommand{\tens}{\otimes}
\newcommand{\id}{\mathrm{id}}
\newcommand{\comp}{\!\circ\!}
\newcommand{\I}{\mathds{1}}
\newcommand{\hh}[1]{\widehat{#1}}
\newcommand{\GG}{\mathbb{G}}
\newcommand{\HH}{\mathbb{H}}
\newcommand{\XX}{\mathbb{X}}
\newcommand{\YY}{\mathbb{Y}}
\newcommand{\vtens}{\,\bar{\tens}\,}
\newcommand{\uu}{{\scriptscriptstyle\mathrm{u}}}
\newcommand{\op}{{\scriptscriptstyle\mathrm{op}}}
\newcommand{\cI}{\mathcal{I}}
\newcommand{\sA}{\mathsf{A}}
\newcommand{\sD}{\mathsf{D}}
\newcommand{\sM}{\mathsf{M}}
\newcommand{\sN}{\mathsf{N}}
\newcommand{\sH}{\mathsf{H}}
\newcommand{\sJ}{\mathsf{J}}
\newcommand{\gM}{\mathfrak{M}}
\newcommand{\dd}[1]{\widetilde{#1}}
\newcommand{\is}[2]{\left(#1\,\vline\,#2\right)}
\DeclareMathOperator{\C}{C}
\DeclareMathOperator{\cZ}{\mathcal{Z}}
\DeclareMathOperator{\M}{M}
\DeclareMathOperator{\Mor}{Mor}
\DeclareMathOperator{\B}{B}
\DeclareMathOperator{\cK}{\mathcal{K}}
\DeclareMathOperator{\Linf}{L^\infty\!\!\;}
\DeclareMathOperator{\Ltwo}{L^2\!\!\;}
\DeclareMathOperator{\vN}{vN}

\numberwithin{equation}{section}

\begin{document}

\author{Pawe{\l} Kasprzak}
\address{Department of Mathematical Methods in Physics, Faculty of Physics, University of Warsaw, Poland
and Institute of Mathematics of the Polish Academy of Sciences,
ul.~\'Sniadeckich 8, 00--956 Warsaw, Poland}
\email{pawel.kasprzak@fuw.edu.pl}

\author{Piotr M.~So{\l}tan} \address{Department of Mathematical Methods in Physics, Faculty of Physics, University of Warsaw, Poland}
\email{piotr.soltan@fuw.edu.pl}

\thanks{Supported by National Science Centre (NCN) grant no.~2011/01/B/ST1/05011}

\title{Embeddable Quantum Homogeneous Spaces}


\begin{abstract}
We discuss various notions generalizing the concept of a homogeneous space to the setting of locally compact quantum groups. On the von Neumann algebra level we find an interesting duality for such objects. A definition of a quantum homogeneous space is proposed along the lines of the pioneering work of Vaes on induction and imprimitivity for locally compact quantum groups. The concept of an embeddable quantum homogeneous space is selected and discussed in detail as it seems to be the natural candidate for the quantum analog of classical homogeneous spaces. Among various examples we single out the quantum analog of the quotient of the Cartesian product of a quantum group with itself by the diagonal subgroup, analogs of quotients by compact subgroups as well as quantum analogs of trivial principal bundles.
\end{abstract}

\maketitle

\section{Introduction}\label{intro}

The notion of a homogeneous space of a locally compact group is of fundamental importance in many branches of mathematics. The non-commutative geometric generalization of the theory of locally compact groups was enriched greatly by the paper of S.~Vaes \cite{Vaes}, in which the notion of a closed subgroup and, more importantly a quantum homogeneous space was thoroughly discussed (alongside many other developments). For compact quantum groups these notions were already developed in the PhD thesis of P.~Podle\'s (\cite{podPHD}) and later published in \cite{Pod1}. It was shown in that paper that quantum groups often have fewer subgroups than one would expect. This was, in particular, proved for the quantum $\mathrm{SU}(2)$ group which is a deformation of the classical $\mathrm{SU}(2)$, yet whose list of subgroups is dramatically shorter than that of the classical $\mathrm{SU}(2)$. Thus Podle\'s realized that some quantum homogeneous spaces did not come from quantum subgroups. It led him to introduce the notions of
\begin{itemize}
\item[$\blacktriangleright$] homogeneous space,
\item[$\blacktriangleright$] embeddable homogeneous space,
\item[$\blacktriangleright$] quotient homogeneous space
\end{itemize}
in the compact quantum group context and he proved that these three classes are consecutively strictly smaller. Moreover in his work on quantum spheres (\cite{spheres}) he showed that by allowing (quantum) homogeneous spaces which are not of the quotient type we reveal a wealth of new examples of interesting quantum spaces. The three different classes of (quantum) homogeneous spaces mentioned above are reduced to usual homogeneous spaces provided the quantum group \emph{and} the homogeneous space are classical. It is important to note that the first (broadest) class allows non-commutative homogeneous spaces for classical groups. However (quantum) homogeneous spaces of the second and third class for a classical group are necessarily classical.

The class of \emph{embeddable} homogeneous spaces was defined by Podle\'s as containing quantum spaces $\XX$ with an action of a compact quantum group $\GG$ which can be realized inside $\C(\GG)$ by the comultiplication. In other words he considered coideals in $\C(\GG)$ (cf. Theorem \ref{uniqC0X}\eqref{uniqC0X2} and Proposition \ref{Wst-emb}). The classical correspondence between closed subgroups and coideals has by now received more attention from researchers in the theory of quantum groups. We would like to point out an interesting approach, due to several authors, using \emph{idempotent states} (see e.g.~\cite{FSalg,FS,SaSa,Sal}). The case of co-amenable compact quantum groups was treated from this point of view in \cite[Theorem 4.1]{FS}. A similar result for unimodular co-amenable locally compact quantum groups can be found in \cite[Theorem 3.5]{SaSa}.

This paper is devoted to generalization of the notion of embeddable homogeneous spaces for compact quantum groups to the context of locally compact quantum groups of Kustermans and Vaes \cite{KV,KVvN,mnw}. We will be following the path begun by Vaes in \cite{Vaes}, where he dealt with the generalization to the non-compact case of the quotient construction considered by Podle\'s.

The task seems quite a lot more complicated than for compact quantum groups. We encounter many different classes of objects related to a given locally compact quantum group. We will introduce all these notions in Sections \ref{prelSect}, \ref{wstSect} and \ref{qhsSect}. Nevertheless, before giving precise definitions, we wish to present a diagram describing (informally) relations between the various concepts which will be dealt with in the paper. This is done in Figure \ref{figQHS}.

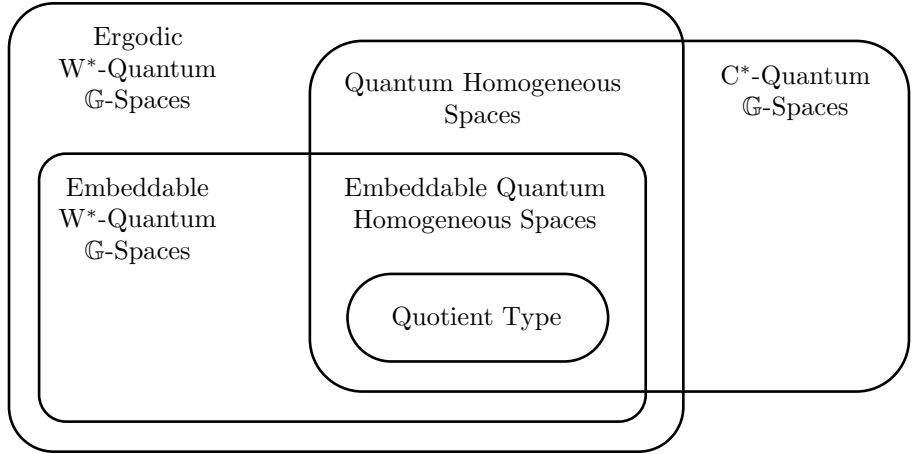
\begin{figure}[h]
\begin{center}
\begin{pspicture}(0,0)(12,6)\label{figQHS}
\psframe[linewidth=1pt,framearc=.2](0,0)(9,6)
\psframe[linewidth=1pt,framearc=.3](4,0.8)(12,5.5)
\psframe[linewidth=1pt,framearc=.2](.4,0.4)(8.5,4)
\psframe[linewidth=1pt,framearc=1](4.5,1.2)(8,2.4)
\put(.5,5){$\begin{array}{c}
\text{Ergodic}\\
\text{\wst-Quantum}\\
\GG\text{-Spaces}
\end{array}$}
\put(4.3,4.6){$\begin{array}{c}
\text{Quantum Homogeneous}\\
\text{Spaces}
\end{array}$}
\put(4.3,3.2){$\begin{array}{c}
\text{Embeddable Quantum}\\
\text{Homogeneous Spaces}\end{array}$}
\put(.5,3){$\begin{array}{c}
\text{Embeddable}\\
\text{\wst-Quantum}\\
\GG\text{-Spaces}
\end{array}$}
\put(9.3,4.7){$\begin{array}{c}
\text{\cst-Quantum}\\
\GG\text{-Spaces}
\end{array}$}
\put(5.1,1.7){Quotient Type}
\end{pspicture}
\end{center}
\caption{Schematic presentation of relations between types of quantum group actions.}
\end{figure}

The intersection between (ergodic) \wst-quantum $\GG$-spaces and \cst-quantum $\GG$-spaces in Figure \ref{figQHS} should not be understood literally. It represents the class of \wst-quantum $\GG$-spaces for which there exists a suitably compatible ``\cst-version'' (see Definition \ref{DefQHS}).

One of the reasons for considering so many different objects is that some new constructions are natural and relatively easy to perform on one level (e.g.~the \wst-level) and seem quite complicated if not impossible on other levels (e.g.~the \cst-level). This is exemplified in particular in Section \ref{wstSect} where we find a very satisfying duality (in the spirit of \cite{TT}) between embeddable \wst-quantum $\GG$-spaces. This duality produces for each embeddable \wst-quantum $\GG$-space $\XX$ an embeddable \wst-quantum $\hh{\GG}$-space which we denote by $\dd{\XX}$ and call the \emph{co-dual} of $\XX$. We also show that the second co-dual $\dd{\dd{\XX}}$ is equal to $\XX$ (this is a true equality, not isomorphism --- the price we pay for this is that our objects come with a particular embedding into operators on appropriate Hilbert spaces). This duality, when restricted to a classical group $\GG$, coincides with the well-known Takesaki-Tatsuuma duality between invariant subalgebras of $\Linf(\GG)$ and closed subgroups of $\GG$ (see \cite{TT}). Since our duality gives rise to the duality between invariant subalgebras of $\Linf(\GG)$ and $\Linf(\hh{\GG})$ it should be stressed that it is the co-commutativity of $\hh{\GG}$ that forces all invariant subalgebras of $\Linf(\hh{\GG})$ to be of the form $\Linf(\hh{\HH})$ for certain closed quantum subgroup $\HH\subset\GG$. In the case of non-classical group $\GG$ we get a fully symmetric picture of the duality that links embeddable \wst-quantum $\GG$-spaces and embeddable \wst-quantum $\hh{\GG}$-spaces.

The necessary material on quantum groups and the operator algebra techniques we use can be found in \cite{KV,unbo}. The research presented in this paper is based strongly on the paper of Vaes \cite{Vaes}, the work on quantum subgroups \cite{subgroups} (which, in turn, follows \cite{mrw} closely) and \cite{certain}. Some important definitions and concepts will be touched upon in Section \ref{prelSect}.

Section \ref{wstSect} introduces the notion of an embeddable \wst-quantum $\GG$-space and deals with co-duality for such objects. Several examples including classical ones are discussed. The analysis of the classical situation reproduces to a certain extent the results of \cite{Salmi}. Then, in Section \ref{qhsSect}, we introduce quantum homogeneous spaces and embeddable quantum homogeneous spaces which are the main focus of this paper. Section \ref{ccSect} deals with what we call the ``co-compact'' cases which are quantum analogues of quotients by compact subgroups (cf.~\cite{SalmiCpt} for another point of view on this situation). In Section \ref{diagSect} the quantum analog of the quotient of the Cartesian product of a group with itself by the diagonal subgroup is discussed. We construct this object and show that it is an embeddable quantum homogeneous space, as defined in Section \ref{qhsSect}. We also prove that it is not of quotient type unless the considered quantum group is, in fact, classical. In other words, the diagonal subgroup does not exist for quantum groups. Section \ref{strongSect} discusses a strengthening of the definition of an embeddable quantum homogeneous space which corresponds to the classical notion of a trivial principal bundle.

\section{Preliminaries}\label{prelSect}

Throughout the paper we will be using the standard tools of the theory of operator algebras employed in the study of quantum groups. Introductory material on these topics may be found e.g.~in \cite{KVoa}. For more advanced topics on multipliers we refer to \cite{unbo,lance}. A crucial notion of a \emph{strict mapping} introduced in \cite[Definition 3.1]{Vaes} will be used extensively. Let $\sN$ be a von Neumann algebra and let $\sA$ be a \cst-algebra. A mapping $\theta\colon\sN\to\M(\sA)$ is \emph{strict} if for any norm-bounded net $(x_i)_{i\in\cI}$ of elements of $\sN$ converging in the strong${}^*$ topology to $y\in\sN$ the net $\bigl(\theta(x_i)\bigr)_{i\in\cI}$ is strictly convergent to $\theta(y)$. This means that for any $a\in\sA$ the nets
\[
\bigl(\theta(x_i)a\bigr)_{i\in\cI}\quad\text{and}\quad\bigl(\theta(x_i)^*a\bigr)_{i\in\cI}
\]
converge in norm to $\theta(y)a$ and $\theta(y)^*a$ respectively.

We will use the very convenient notation which can be found e.g.~in \cite[Section 2]{BSV}: for a subset $\mathcal{X}$ of a Banach space the symbol $[\mathcal{X}]$ denotes the norm-closed linear span of $\mathcal{X}$. To keep the notation light we also write $[\,\cdots]$ instead of $[\{\,\cdots\}]$ whenever necessary.

Let us recall the definition of a locally compact quantum group on von Neumann algebra level:

\begin{definition}[{\cite[Definition 1.1]{KVvN}}]\label{lcqg}
A pair $\GG=(\sM,\Delta)$ consisting of a von Neumann algebra $\sM$ and a normal unital $*$-homomorphism $\Delta\colon\sM\to\sM\vtens\sM$ is called a \emph{locally compact quantum group} if
\begin{itemize}
\item[$\blacktriangleright$] $\Delta$ is \emph{coassociative}, i.e.~$(\Delta\tens\id)\circ\Delta=(\id\tens\Delta)\circ\Delta$;
\item[$\blacktriangleright$] there exist n.s.f.~weights $\varphi$ and $\psi$ on $\sM$ such that
\begin{itemize}
\item[-] $\varphi$ is left invariant: $\varphi\bigl((\omega\tens\id)\Delta(x)\bigr)=\varphi(x)\omega(\I)$ for all $x\in\gM_{\varphi}^+$ and $\omega\in\sM_*^+$,
\item[-] $\psi$ is right invariant: $\psi\bigl((\id\tens\omega)\Delta(x)\bigr)=\psi(x)\omega(\I)$ for all $x\in\gM_{\psi}^+$ and $\omega\in\sM_*^+$.
\end{itemize}
\end{itemize}
\end{definition}

In what follows we shall write $\Linf(\GG)$ for the von Neumann algebra $\sM$ from Definition \ref{lcqg} and $\Delta_{\GG}$ for $\Delta$. There are a few possible conventions available while developing general theory. The choices correspond to the left and right Haar weight option. Anticipating the needs of this paper we adopt the following choices and notation (found also in \cite{BS, mnw}):
\begin{itemize}
\item[$\blacktriangleright$] $\Ltwo(\GG)$ denotes the GNS-Hilbert space for the right Haar weight $\psi$; in what follows $\eta$ will denote the corresponding GNS-map.
\item[$\blacktriangleright$] $W\in\B\bigl(\Ltwo(\GG)\tens\Ltwo(\GG)\bigr)$ is the \emph{Kac-Takesaki} operator of $\GG$. It is a uniquely defined by its values on the simple tensors $\eta(x)\tens\eta(y)$:
\[
W\bigl(\eta(x)\tens\eta(y)\bigr)=(\eta\tens\eta)\bigl(\Delta_{\GG}(x)(\I\tens{y})\bigr).
\]
\item[$\blacktriangleright$] The Tomita-Takesaki antiunitary conjugation related to $\psi$ is denoted by $J$.
\end{itemize}
Starting with $\GG$ one constructs the dual quantum group $\hh{\GG}$. By definition
\[
\Linf(\hh{\GG})=\bigl\{(\id\tens\omega)(W)\st\,\omega\in\B(\Ltwo(\GG))_*\bigr\}''
\]
and $\Ltwo(\hh{\GG})=\Ltwo(\GG)$. The last equality means that the range of the GNS-map $\hh{\eta}$ of the right Haar weight $\hh{\psi}$ of $\hh{\GG}$ is $\Ltwo(\GG)$ and we have $\hh{\psi}(x^*x)=\is{\hh{\eta}(x)}{\hh{\eta}(x)}$ for any $x$ in the domain of $\hh{\eta}$. The related modular conjugation $\hh{J}$ is an antiunitary operator acting on $\Ltwo(\GG)$. Remarkably, $\hh{J}$ implements the \emph{unitary antipode} $R$ of $\GG$:
\[
R(x)=\hh{J}x^*\hh{J}\text{ for any }x\in\Linf(\GG)
\]
(\cite[section 5]{KV}). The Kac-Takesaki operator $\hh{W}\in\B\bigl(\Ltwo(\GG)\tens\Ltwo(\GG)\bigr)$ of $\hh{\GG}$ is, in turn, given by $\hh{W}=\Sigma{W^*}\Sigma$, where $\Sigma\in\B\bigl(\Ltwo(\GG)\tens\Ltwo(\GG)\bigr)$ denotes the flip operator.

The quantum group $\GG$ can be equivalently described using the language of \cst-algebras. In particular, given a locally compact quantum group $\GG$ there is a \cst-algebra, which we will denote by $\C_0(\GG)$, strongly dense in $\Linf(\GG)$ such that $\Delta_{\GG}$ restricted to $\C_0(\GG)$ is a morphism from $\C_0(\GG)$ to $\C_0(\GG)\tens\C_0(\GG)$ in the sense of \cite{unbo} (see also \cite{lance}). We will use the same symbol (namely $\Delta_{\GG}$) for comultiplications on $\Linf(\GG)$ and $\C_0(\GG)$ as well as its extension to a map $\M\bigl(\C_0(\GG)\bigr)\to\M\bigl(\C_0(\GG)\tens\C_0(\GG)\bigr)$.

\begin{definition}
Let $\GG$ be a locally compact quantum group. A \emph{\wst-quantum $\GG$-space} $\XX$ consists of a von Neumann algebra $\Linf(\XX)$ and an injective normal unital $*$-homomorphism $\delta_\XX\colon\Linf(\XX)\to\Linf(\GG)\vtens\Linf(\XX)$ such that
\[
(\id\tens\delta_\XX)\comp\delta_\XX=(\Delta_{\GG}\tens\id)\comp\delta_\XX.
\]
A \wst-quantum $\GG$-space $\XX$ is \emph{ergodic} if $\delta_\XX(x)=\I\tens{x}$ implies $x\in\CC\I$.
\end{definition}

\begin{definition}
Let $\GG$ be a locally compact quantum group. A \emph{\cst-quantum $\GG$-space} $\YY$ consists of a \cst-algebra $\C_0(\YY)$ and an injective
\[
\delta_\YY\in\Mor\bigl(\C_0(\YY),\C_0(\GG)\tens\C_0(\YY)\bigr)
\]
such that
\[
(\id\tens\delta_\YY)\comp\delta_\YY=(\Delta_{\GG}\tens\id)\comp\delta_\YY.
\]
and the Podle\'s condition holds:
\[
\bigl[\delta_\YY\bigl(\C_0(\YY)\bigr)\bigl(\C_0(\GG)\tens\I\bigr)\bigr]=\C_0(\GG)\tens\C_0(\YY).
\]
\end{definition}

\section{Embeddable \wst-quantum $\GG$-spaces}\label{wstSect}

\begin{proposition}\label{Wst-emb}
Let $\GG$ be a locally compact quantum group and let $\XX$ be an ergodic \wst-quantum $\GG$-space. Assume that there exists a normal unital $*$-homomorphism $\gamma\colon\Linf(\XX)\to\Linf(\GG)$ such that
\begin{equation}\label{intw}
(\id\tens\gamma)\comp\delta_\XX=\Delta_{\GG}\comp\gamma.
\end{equation}
Then $\gamma$ is injective.
\end{proposition}

\begin{proof}
Let $\sJ\subset\sN$ be the kernel of $\gamma$ and $p\in\cZ(N)$ be the central projection corresponding to the two-sided ideal $\sJ$: $\sJ=p\sN$. Equation \eqref{intw} implies that $\delta_\XX(\sJ)\subset\Linf(\GG)\vtens\sJ$ which, in turn, leads to the inequality $\delta_\XX(p)\leq\I\tens{p}$. It has been shown in the proof of \cite[Theorem 4.2]{certain} that in fact we have $\delta_\XX(p)=\I\tens{p}$. Using the ergodicity of $\XX$ we conclude that either $p=0$ or $p=\I$. The second case is ruled out, since $\gamma$ is a unital map. Hence $p=0$ and $\ker\gamma=\sJ=\{0\}$.
\end{proof}

Proposition \ref{Wst-emb} says that if we are given a \wst-quantum $\GG$-space $\XX$ with an equivariant map $\gamma\colon\Linf(\XX)\to\Linf(\GG)$ then we may regard $\Linf(\XX)$ as a right coideal in $\Linf(\GG)$. Therefore, in such a situation, we will from now on assume that the embedding is part of the data.

\begin{definition}\label{EWstGs}
Let $\GG$ be a locally compact quantum group and $\XX$ an ergodic \wst-quantum $\GG$-space. We say that $\XX$ is an \emph{embeddable \wst-quantum $\GG$-space} if $\Linf(\XX)\subset\Linf(\GG)$ and $\delta_\XX$ is given by restriction of $\Delta_{\GG}$ to $\Linf(\XX)$. In particular we have $\Delta_{\GG}\bigl(\Linf(\XX)\bigr)\subset\Linf(\GG)\vtens\Linf(\XX)$.
\end{definition}

\begin{example}\label{wstEx1}
One obvious example of an embeddable \wst-quantum $\GG$-space is $\XX=\GG$. More generally, let $\HH$ be a closed quantum subgroup of $\GG$ in the sense Vaes (\cite{Vaes,subgroups}, cf.~also \cite{VV}). Performing the quotient construction one gets an embeddable \wst-quantum $\GG$-space $\GG/\HH$, see \cite[Definition 4.1]{Vaes}. Examples related to classical groups will be discussed in more detail in Theorem \ref{classX}.
\end{example}

\begin{definition}\label{defqtw} Let $\XX$ be an embeddable \wst-quantum $\GG$-space. We say that $\XX$ is of quotient type if there exists a closed quantum subgroup in the sense of Vaes $\HH$ of $\GG$ such that $\Linf(\XX)=\Linf(\GG/\HH)$.
\end{definition}

\begin{proposition}\label{propDuYY}
Let $\XX$ be an embeddable \wst-quantum $\GG$-space and let $\sN$ be the relative commutant of $\Linf(\XX)$ in $\Linf(\hh{\GG})$:
\[
\sN=\bigl\{y\in\Linf(\hh{\GG})\st\;\forall\:x\in\Linf(\XX)\,\:xy=yx\bigr\}.
\]
Then
\begin{enumerate}
\item\label{propDuYY1} $\Delta_{\hh{\GG}}(\sN)\subset\Linf(\hh{\GG})\vtens\sN$,
\item\label{propDuYY2} the pair $(\sN,\gamma)$ with $\gamma=\bigl.\Delta_{\hh{\GG}}\bigr|_\sN$ is an embeddable \wst-quantum $\hh{\GG}$-space.
\end{enumerate}
\end{proposition}

\begin{proof}
Clearly \eqref{propDuYY2} follows from \eqref{propDuYY1}. Now take $y\in\sN$ and $x\in\Linf(\XX)$. Let $W\in\B\bigl(\Ltwo(\GG)\tens\Ltwo(\GG)\bigr)$ be the Kac-Takesaki operator. We have
\[
\begin{split}
\Delta_{\hh{\GG}}(y)(\I\tens{x})&=\hh{W}(y\tens\I)\hh{W}^*(\I\tens{x})\\
&=\Sigma{W^*}\Sigma(y\tens\I)\Sigma{W}\Sigma(\I\tens{x})\Sigma{W^*W}\Sigma\\
&=\Sigma{W^*}(\I\tens{y}){W}(x\tens\I){W^*W}\Sigma\\
&=\Sigma{W^*}(I\tens{y})\Delta_{\GG}(x)W\Sigma\\
&=\Sigma{W^*}\Delta_{\GG}(x)(\I\tens{y})W\Sigma\\
&=(\I\tens{x})\Sigma{W^*}(\I\tens{y})W\Sigma\\
&=(\I\tens{x})\Delta_{\hh{\GG}}(y).
\end{split}
\]
This shows that $\Delta_{\hh{\GG}}(y)\in\Linf(\hh{\GG})\vtens\sN$.
\end{proof}

\begin{definition}
Let $\XX$ be an embeddable \wst-quantum $\GG$-space. The embeddable \wst-quantum $\GG$-space $(\sN,\gamma)$ defined in Proposition \ref{propDuYY} will be called the \emph{co-dual} of $\XX$.
\end{definition}

From now on the co-dual embeddable \wst-quantum $\GG$-space of $\XX$ will be denoted by $\dd{\XX}$. In particular we have
\begin{equation}\label{corem}
\Linf(\dd{\XX})=\bigl\{y\in\Linf(\hh{\GG})\st\;\forall\:x\in\Linf(\XX)\,\:xy=yx\bigr\}.
\end{equation}
Let us also mention that a trace of the duality for embeddable \wst-quantum $\GG$-spaces can be found in \cite[Section 4]{FSalg}.

\begin{example}
Consider a locally compact quantum group $\GG$ as an embeddable \wst-quantum $\GG$-space (cf.~Example \ref{wstEx1}). Then the co-dual $\dd{\GG}$ is a one point set which can be naturally identified with the (classical) trivial subgroup of $\hh{\GG}$ (cf.~Theorem \ref{classX}).
\end{example}

\begin{remark}\label{cpRem}
Let $\XX$ be an embeddable \wst-quantum $\GG$-space and define $\sM$ to be the von Neumann algebra
\begin{equation}\label{crossprd}
\sM=(\Linf(\XX)\cup\Linf(\hh{\GG})'\bigr)''\subset\B\bigl(\Ltwo(\GG)\bigr).
\end{equation}
This is the \emph{crossed product} von Neumann algebra as defined \cite[Definition 2.1]{unitImpl}. Usually in this context the crossed product is defined as a subalgebra of $\B\bigl(\Ltwo(\GG)\bigr)\vtens\Linf(\XX)$ generated by the image of $\delta_\XX$ and the copy $\Linf(\hh{\GG})'\tens\I$ of $\Linf(\hh{\GG})'$, but since $\delta_\XX$ is the restriction of $\Delta_{\GG}$ we can use the operator $W$ to bring this subalgebra back to $\B\bigl(\Ltwo(\GG)\bigr)$. More precisely we note that $W^*(y\tens\I)W=y\tens\I$ for any $y\in\Linf(\hh{\GG})'$ and $W^*\delta_\XX(x)W=x\tens\I$ for any $x\in\Linf(\XX)$. Therefore we can identify $\sM$ with the crossed product (cf.~remarks after \cite[Definition 4.1]{unitImpl}). Note that since we are using the right Haar measure, $\Linf(\hh{\GG})$ is the quantum counterpart of the algebra generated by the \emph{right shifts} on $\GG$. Following this analogy we see that the commutant $\Linf(\hh{\GG})'$ is the ``algebra generated by the left shifts'' and its appearance in the crossed product \eqref{crossprd} is compatible with $\XX$ being a left quantum $\GG$-space.

Now we easily see that by definition $\Linf(\dd{\XX})=\Linf(\hh{\GG})\cap\Linf(\XX)'=\sM'$. Moreover we clearly have
\begin{equation}\label{ddX}
\Linf(\dd{\dd{\XX}})=\Linf(\GG)\cap\Linf(\dd{\XX})'=\Linf(\GG)\cap\sM.
\end{equation}
\end{remark}

\begin{theorem}\label{ddddT}
Let $\XX$ be an embeddable \wst-quantum $\GG$-space. Then $\dd{\dd{\XX}}=\XX$.
\end{theorem}

\begin{proof}
Let $J$ and $\hh{J}$ be modular conjugations for the right Haar measures of $\GG$ and $\hh{\GG}$ respectively. Also let $\sM$ be the crossed product algebra as described in Remark \ref{cpRem}. Our aim is to show that
\begin{equation}\label{aim}
\Linf(\GG)\cap\sM=\Linf(\XX).
\end{equation}
(cf.~\eqref{ddX}).

Consider the unitary operator
\[
U=(\I\tens\hh{J}J)W^*(\I\tens{J}\hh{J}),
\]
where $W\in\Linf(\hh{\GG})\vtens\Linf(\GG)$ is the Kac-Takesaki operator. Note that we have
\[
U\in\Linf(\hh{\GG})\vtens\Linf(\GG)'
\]
and the slices of the first leg of $U$ generate $\Linf(\GG)'$ (cf.~\cite[Proposition 2.15]{KVvN}). Therefore equality \eqref{ddX} implies that
\begin{equation}\label{equivinv}
x\in\Linf(\dd{\dd{\XX}})\quad\Longleftrightarrow\quad{x}\in\sM\text{ and }U(\I\tens x)U^*=\I\tens{x}.
\end{equation}
Consider the mapping $\alpha\colon{x}\mapsto{U}(\I\tens{x})U^*$ defined on $\sM$. We have $\alpha\colon\sM\to\Linf(\hh{\GG})\vtens\sM$. Indeed, clearly
\begin{equation}\label{xinMalph}
\alpha(x)=\I\tens{x}\in\Linf(\hh{\GG})\vtens\sM
\end{equation}
for $x\in\Linf(\XX)\subset\sM$, while for $y\in\Linf(\hh{\GG})$ we have
\[
\begin{split}
U(\I\tens\hh{J}y\hh{J})U^*
&=(\I\tens\hh{J}J)W^*(\I\tens{JyJ})W(\I\tens{J}\hh{J})\\
&=(\I\tens\hh{J}J)\Sigma\Delta_{\hh{\GG}}\bigl(\hh{R}(y^*)\bigr)\Sigma(\I\tens{J}\hh{J})\\
&=(\I\tens\hh{J}J)(\hh{R}\tens\hh{R})\Delta_{\hh{\GG}}(y^*)(\I\tens{J}\hh{J})\\
&=(J\tens\hh{J})\Delta_{\hh{\GG}}(y)(J\tens\hh{J}),
\end{split}
\]
so that for $z\in\Linf(\hh{\GG})'$ we also have $\alpha(z)\in\Linf(\hh{\GG})\vtens\Linf(\hh{\GG})'\subset\Linf(\hh{\GG})\vtens\sM$ (cf.~again \cite[Proposition 2.15]{KVvN} or \cite[Section 5]{mnw}). In fact $\alpha$ is the \emph{dual action} of $\hh{\GG}$ on the crossed product $\sM$ as defined in \cite[Section 2]{unitImpl} (cf.~also \cite[Proposition 2.2]{unitImpl}). By \cite[Theorem 2.7]{unitImpl} the \emph{fixed point algebra} $\sM^\alpha$ (\cite[Definition 1.2]{unitImpl}) is equal to the canonical copy of $\Linf(\XX)$ inside $\sM$. Since \eqref{equivinv} may be rephrased as $x\in\dd{\dd{\XX}}\Leftrightarrow{x}\in\sM^\alpha$, we see that $\Linf(\dd{\dd{\XX}})=\sM^\alpha=\Linf(\XX)$.
\end{proof}

Classical duality results in non-commutative harmonic analysis (\cite{TT}) together with the bi-co-duality result of Theorem \ref{ddddT} give us the following description of embeddable \wst-quantum $\GG$-spaces for a \emph{classical} $\GG$.

\begin{theorem}\label{classX}
Let $G$ be a locally compact group and $\GG$ the associated locally compact quantum group (with commutative $\Linf(\GG)=\Linf(G)$). Let $\XX$ be an embeddable \wst-quantum $\GG$-space. Then there exists a closed subgroup $H\subset{G}$ such that
\begin{enumerate}
\item\label{classX1} $\Linf(\XX)=\Linf(G/H)$,
\item\label{classX2} $\Linf(\dd{\XX})=\Linf(\hh{\HH})\subset\Linf(\hh{\GG})$,
\end{enumerate}
where $\HH$ is the locally compact quantum group associated to $H$ and $\hh{\HH}$ its dual.
\end{theorem}

Note that we identified $\Linf(\hh{\HH})$ with its image in $\Linf(\hh{\GG})$.

\begin{proof}[Proof of theorem \ref{classX}]
\eqref{classX1} is a direct consequence of Theorem 5 and Theorem 2 of \cite{TT} adapted to our situation (in particular Takesaki and Tatsuuma consider what we would call \emph{right actions}). Moreover we have
\[
\Linf(\XX)=\bigl\{f\in\Linf(G)\st{f}\text{ is constant on the right cosets of }H\bigr\}
\]
which is equal to the intersection of $\Linf(G)$ with the commutant of the subalgebra of $\B\bigl(\Ltwo(G)\bigr)$ generated by right shifts by elements of $H$. This is saying manifestly that if we define $\YY$ by
\[
\Linf(\YY)=\vN(H)
\]
(the algebra generated by right shifts, see \cite[Section 4]{subgroups}) then $\XX=\dd{\YY}$. By Theorem \ref{ddddT} $\dd{\XX}=\YY$ which is precisely \eqref{classX2}.
\end{proof}

The next lemma provides a characterization of embeddable \wst-quantum $\GG$-spaces of quotient type in terms of their co-duals.

\begin{lemma}\label{charQT}
Let $\XX$ be an embeddable \wst-quantum $\GG$-space. Then $\XX$ is of quotient type if there exists a closed quantum subgroup $\HH$ of $\GG$ in the sense of Vaes such that $\Linf(\dd\XX)$ is the image of $\Linf(\hh{\HH})$ in $\Linf(\hh{\GG})$.
\end{lemma}

\begin{proof}
Let $\HH$ be a closed quantum subgroup of $\GG$ in the sense of Vaes with corresponding injective $*$-homomorphism $\gamma\colon\Linf(\hh\HH)\to\Linf(\hh\GG)$ (see \cite[Theorem 3.3]{subgroups}). Let $W^\HH\in\Linf(\hh\HH)\vtens\Linf(\HH)$ be the Kac-Takesaki operator of $\HH$ and let $\alpha\colon\Linf(\GG)\to\Linf(\GG)\vtens\Linf(\HH)$ be the corresponding action of $\HH$ on $\GG$:
\[
\alpha(x)=\bigl((\gamma\otimes\id)(W^\HH)\bigr)(x\tens\I_{\Linf(\hh\HH)})\bigl((\gamma\otimes\id)(W^{\HH})^*\bigr).
\]
Since $x\in\Linf(\GG/\HH)$ if and only if $\alpha(x)=x\tens\I$ (i.e.~$(\gamma\otimes\id)(W^\HH)$ commutes with $x\tens\I_{\Linf(\hh\HH)}$) and slices of the second leg of $W^\HH$ generate $\Linf(\hh\HH)$, we see that $x\in\Linf(\GG/\HH)$ if $xy=yx$ for any $y\in\gamma\bigl(\Linf(\hh\HH)\bigr)$. By \eqref{corem} we conclude that $\GG/\HH$ is the co-dual of $\dd{\XX}$.
\end{proof}

The above lemma suggests considering embeddable \wst-quantum $\GG$-spaces $\XX$ such that co-dual $\dd\XX$ satisfies $\delta_{\dd{\XX}}\bigl(\Linf(\dd{\XX})\bigr)\subset\Linf(\dd{\XX})\vtens\Linf(\dd{\XX})$. In particular, it is important to characterize the situations when $\dd{\XX}$ actually corresponds to a closed quantum subgroup of $\GG$. A relevant tool might be provided by \cite[Proposition 10.5]{BV}.

\section{Quantum homogeneous spaces \& embeddable quantum homogeneous spaces}\label{qhsSect}

In this section we introduce the definitions of a quantum homogeneous space and an embeddable quantum homogeneous space. As mentioned in Section \ref{intro} they are objects with two complementary descriptions: a von Neumann algebraic one and a \cst-algebraic one. This idea is taken directly from the pioneering work of Vaes in \cite[Section 6]{Vaes}, where the author considers quantum homogeneous spaces of quotient type (for regular locally compact quantum groups). His results motivated a general definition given in \cite[Definition 3.1]{certain} which we recall below and then specialize to the embeddable case in the further part of this section.

\begin{definition}\label{DefQHS}
Let $\GG$ be a locally compact quantum group. A \emph{quantum homogeneous space} for $\GG$ is an ergodic \wst-quantum $\GG$-space $\XX$ such that there exists a \cst-quantum $\GG$-space $\YY$ with
\begin{enumerate}
\item $\C_0(\YY)$ is a strongly dense \cst-subalgebra of $\Linf(\XX)$,
\item $\delta_{\YY}$ is given by the restriction of $\delta_\XX$ to $\C_0(\YY)$,
\item $\delta_\XX\bigl(\Linf(\XX)\bigr)\subset\M\bigl(\cK(\Ltwo(\GG))\tens\C_0(\YY)\bigr)$ and the map
\[
\delta_\XX\colon\Linf(\XX)\longrightarrow
\M\bigl(\cK(\Ltwo(\GG))\tens\C_0(\YY)\bigr)
\]
is strict.
\end{enumerate}
\end{definition}

In the situation described in Definition \ref{DefQHS} we will identify the ``quantum spaces'' $\XX$ and $\YY$ and simply write $\Linf(\XX)$ and $\C_0(\XX)$ for the corresponding von Neumann algebra and \cst-algebra respectively. We will also say that $\C_0(\XX)$ together with $\delta_\XX\in\Mor\bigl(\C_0(\XX),\C_0(\GG)\tens\C_0(\XX)\bigr)$ describes the \cst-version of $\XX$. This is justified by the fact that the \cst-algebra $\C_0(\XX)$ is unique (see Theorem \ref{uniqC0X} below). Another point of view on the concept of quantum homogeneous space is that we have one ``quantum space'' $\XX$ with two structures: topological described by $\C_0(\XX)$ and measurable described by $\Linf(\XX)$.

\begin{definition}\label{DefEQHS}
Let $\GG$ be a locally compact quantum group and let $\XX$ be a quantum homogeneous space of $\GG$. We say that $\XX$ is embeddable if the \wst-version of $\XX$ is an embeddable \wst-quantum $\GG$-space.
\end{definition}

\begin{example}
\noindent\begin{enumerate}
\item Let $\GG$ be a locally compact quantum group. The quantum group itself is a \wst-quantum $\GG$-space. The fact that $\GG$ has its \cst-algebraic version (\cite[Proposition 1.7]{KVvN}) means that $\GG$ is a quantum homogeneous space as defined in Definition \ref{DefQHS}.
\item Let $\GG$ be a locally compact quantum group (\cite{BS,BSV}) and let $\HH$ be closed subgroup of $\GG$ in the sense of Vaes (\cite{Vaes,subgroups}). The construction of $\Linf(\GG/\HH)$ may be easily performed by taking the fixed point algebra for an action of $\HH$ on $\Linf(\GG)$. However the existence of $\C_0(\GG/\HH)$ is a non-trivial matter. It was proved in \cite[Theorem 6.1]{Vaes} under the regularity assumption on $\GG$. Actually Vaes notes that the regularity might not be necessary for the existence of the quantum homogeneous space $\GG/\HH$ as was for example shown for the case of $\HH$ compact in \cite{Psol} (cf.~Example \ref{ccEx}\eqref{ccEx2}).
\end{enumerate}
\end{example}

\begin{definition}\label{qtype}
Let $\GG$ be a locally compact quantum group and let $\XX$ be a quantum homogeneous space of $\GG$. We say that $\XX$ is of \emph{quotient type} if the \wst-version of $\XX$ is of quotient type in the sense of Definition \ref{defqtw}.

This terminology agrees with the original one introduced by Podle\'s \cite{Pod1,podPHD} although he used a slightly more restrictive notion of a quantum subgroup.
\end{definition}

\begin{theorem}[{\cite[Propositions 3.4 \& 3.5]{certain}}]\label{uniqC0X}
Let $\GG$ be a locally compact quantum group and let $\XX$ be a quantum homogeneous space for $\GG$. Then
\begin{enumerate}
\item the \cst-algebra $\C_0(\XX)$ is uniquely determined by the conditions of Definition \ref{DefQHS},
\item\label{uniqC0X2} if $\YY$ is another quantum homogeneous space for $\GG$ and $\pi\in\Mor\bigl(\C_0(\XX),\C_0(\YY)\bigr)$ is an equivariant isomorphism then $\pi$ extends to an equivariant isomorphism $\Linf(\XX)\to\Linf(\YY)$.
\end{enumerate}
\end{theorem}

The term ``embeddable quantum homogeneous space'' is perhaps not the best since it is not the quantum space that embeds, but the corresponding von Neumann algebra of functions. Nevertheless such terminology has been functioning for many years in the context of compact quantum groups (\cite[Definition 1.8]{Pod1}). Note also that quantum homogeneous spaces of quotient type (Definition \ref{qtype}) are prototypical examples of embeddable quantum homogeneous spaces.

Let $\XX$ be an embeddable quantum homogeneous space for a locally compact quantum group $\GG$. As with embeddable \wst-quantum $\GG$-spaces we will from now on regard the embedding $\Linf(\XX)\hookrightarrow\Linf(\GG)$ as part of the data. In particular we will simply view $\Linf(\XX)$ as a right coideal in $\Linf(\GG)$ and $\delta_\XX$ will be identified with $\bigl.\Delta_{\GG}\bigr|_{\Linf(\XX)}$. Nevertheless the symbol $\delta_\XX$ will be used (see e.g.~proof of Proposition \ref{morfizm}).

Before we continue let us note the following:

\begin{proposition}\label{classG}
Let $G$ be a locally compact group. The class of embeddable quantum homogeneous spaces for $G$ considered as a locally compact quantum group coincides with the class of classical homogeneous spaces for $G$.
\end{proposition}

The proof of Proposition \ref{classG} is immediate e.g.~from Theorem \ref{classX}.

Let us now turn attention to the \cst-algebra $\C_0(\XX)$. We have $\C_0(\XX)\subset\Linf(\XX)\subset\Linf(\GG)$. The situation is clarified further by the next proposition.

\begin{proposition}\label{morfizm}
Let $\GG$ be a locally compact quantum group and let $\XX$ be an embeddable quantum homogeneous space for $\GG$. Then
\begin{enumerate}
\item\label{morfizm1} $\C_0(\XX)\subset\M\bigl(\C_0(\GG)\bigr)$,
\item\label{morfizm2} the embedding $\C_0(\XX)\hookrightarrow\M\bigl(\C_0(\GG)\bigr)$ is an element of $\Mor\bigl(\C_0(\XX),\C_0(\GG)\bigr)$.
\end{enumerate}
\end{proposition}

\begin{proof}
Let $W\in\M\bigl(\cK\bigl(\Ltwo(\GG)\bigr)\tens\C_0(\GG)\bigr)$ be the Kac-Takesaki operator for $\GG$. Since for any $x\in\C_0(\XX)$ we have $\delta_\XX(x)=\Delta_{\GG}(x)=W(x\tens\I)W^*$, clearly
\begin{equation}\label{eqdel}
\delta_{\XX}\bigl(\C_0(\XX)\bigr)\subset\M\bigl(\cK(\Ltwo(\GG))\tens\C_0(\GG)\bigr).
\end{equation}
Now the Podle\'s condition for $\delta_\XX$ implies that
\[
\C_0(\XX)=\bigl[(\omega\tens\id)\delta_{\XX}(\C_0(\XX))\st\omega\in\B(\Ltwo(\GG))_*\bigr]
\]
which together with \eqref{eqdel} shows that $\C_0(\XX)\subset\M\bigl(\C_0(\GG)\bigr)$ and proves \eqref{morfizm1}.

In order to see that the inclusion $\C_0(\XX)\hookrightarrow\M\bigl(\C_0(\GG)\bigr)$ is a morphism from $\C_0(\XX)$ to $\C_0(\GG)$ we first note that the Podle\'s condition for $\delta_\XX$ implies the following equality:
\begin{equation}\label{da}
\bigl[\C_0(\XX)\C_0(\GG)\bigr]
=\bigl[(\omega\tens\id)(\delta_{\XX}(\C_0(\XX))(\cK(\Ltwo(\GG))
\tens\C_0(\GG)))\st\omega\in\B(\Ltwo(\GG))_*\bigr].
\end{equation}
Let $(e_i)_{i\in\cI}$ be an approximate unit of $\C_0(\XX)$ and let us fix a simple tensor $k\tens{a}\in\cK\bigl(\Ltwo(\GG)\bigr)\tens\C_0(\GG)$. Since for each $i$
\[
\delta_{\XX}(e_i)(k\tens{a})=W(e_i\tens\I)W^*(k\tens{a}),
\]
we find that $\lim\limits_{i\in\cI}\delta_{\XX}(e_i)(k\tens{a})=k\tens{a}$ in norm. In particular for any $a\in\C_0(\GG)$ we have
\[
a\in
\bigl[(\omega\tens\id)(\delta_\XX(\C_0(\XX))(\cK(\Ltwo(\GG))\tens\C_0(\GG)))\st\omega\in\B(\Ltwo(\GG))_*\bigr].
\]
This, together with \eqref{da}, proves that $\C_0(\GG)=\bigl[\C_0(\XX)\C_0(\GG)\bigr]$ which is precisely \eqref{morfizm2}.
\end{proof}

At the end of this section let us show that whenever $\HH$ is a closed quantum subgroup of a locally compact quantum group $\GG$ then $\Linf(\hh{\HH})$ may be interpreted as an embeddable quantum homogeneous $\hh{\GG}$-space with a \cst-version $\C_0(\hh{\HH})$. This result will be used in Section \ref{expbc} in the analysis of a quotient quantum homogeneous spaces arising in the bicrossed product construction.

\begin{proposition}\label{cvhProp}
Let $\HH$ be a closed quantum subgroup of $\GG$ with the corresponding normal inclusion $\gamma\colon\Linf(\hh{\HH})\to\Linf(\hh{\GG})$. Define $\XX$ by putting $\Linf(\XX)=\gamma\bigl(\Linf(\hh{\HH})\bigr)$,and let $\delta_\XX=\bigl.\Delta_{\hh{\GG}}\bigr|_{\Linf(\XX)}$. Then $\XX$ is an embeddable \wst-quantum $\hh{\GG}$-space. Moreover $\XX$ is an embeddable quantum homogeneous space and $\C_0(\XX)=\gamma\bigl(\C_0(\hh{\HH})\bigr)$.
\end{proposition}

\begin{proof}
The fact that $\XX$ is an embeddable \wst-quantum $\hh{\GG}$-space is clear from the property of $\gamma$:
\[
\Delta_{\hh{\GG}}\comp\gamma=(\gamma\tens\gamma)\comp\Delta_{\hh{\HH}}.
\]
We will now show that $\XX$ is in fact an embeddable quantum homogeneous space for $\hh{\GG}$.

Note that since $\gamma$ is normal and injective, there exists a Hilbert space $\sH$ and a unitary operator $T\colon\sH\tens\Ltwo(\HH)\to\sH\tens\Ltwo(\GG)$ such that
\begin{equation}\label{timpl}
\I_{\sH}\tens\gamma(x)=T(\I_{\sH}\tens{x})T^*.
\end{equation}
Let $\hh{W}\in\M\bigl(\C_0(\HH)\tens\C_0(\hh{\HH})\bigr)$ be the Kac-Takesaki operator of $\hh{\HH}$. Let $(y_i)_{i\in\cI}$ be a bounded net of elements of $\Linf(\hh{\HH})$ convergent in the strong${}^*$ topology and take $x\in\cK\bigl(\Ltwo(\GG)\bigr)$ and $d\in\gamma\bigl(\C_0(\hh{\HH})\bigr)$. Define $Y=(\id\tens\gamma)\hh{W}$. In order to check the strictness condition of Definition \ref{DefQHS} we compute
\[
\begin{split}
\Delta_{\hh{\GG}}\bigl(\gamma(y_i)\bigr)(x\tens{d})
&=(\gamma\tens\gamma)\bigl(\Delta_{\hh{\HH}}(y_i)\bigr)(x\tens{d})\\
&=(\gamma\tens\id)\bigl(Y(y_i\tens\I)Y^*\bigr)(x\tens{d}).
\end{split}
\]
Using \eqref{timpl} we see that
\[
\begin{split}
\I_{\sH}\tens\bigl(\Delta_{\hh{\GG}}(\gamma(y_i))(x\tens{d})\bigr)
&=\I_{\sH}\tens\bigl((\gamma\tens\id)(Y(y_i\tens\I)Y^*)(x\tens{d})\bigr)\\
&=T_{12}Y_{23}(\I_{\sH}\tens{y_i}\tens\I)Y^*_{23}T^*_{12}(\I_{\sH}\tens{x}\tens{d}).
\end{split}
\]
Let us fix a $1$-dimensional projection $P\in\cK(\sH)$. Note that since the first leg of all elements of the net
\[
\bigl(T_{12}Y_{23}(\I\tens{y_i}\tens\I)Y^*_{23}T^*_{12}(\I\tens{x}\tens{d})\bigr)_{i\in\cI}
\]
is in $\CC\I_{\sH}$, its norm-convergence is equivalent to norm-convergence of the net
\begin{equation}\label{nr}
\bigl(T_{12}Y_{23}(\I\tens{y_i}\tens\I)Y^*_{23}T^*_{12}(P\tens{x}\tens{d})\bigr)_{i\in\cI}.
\end{equation}
Furthermore, since
\[
Y^*_{23}T^*_{12}(P\tens{x}\tens{d})\in\cK(\sH)\tens\cK\bigl(\Ltwo(\hh{\GG})\bigr)\tens\gamma\bigl(\C_0(\hh{\HH})\bigr),
\]
we see that strong${}^*$-convergence of $(y_i)_{i\in\cI}$ implies norm-convergence of the net \eqref{nr}. This is in turn equivalent to the norm-convergence of the net $\bigl(\Delta_{\hh{\GG}}(\gamma(y_i))(x\tens{d})\bigr)_{i\in\cI}$. Norm-convergence of $\bigl(\Delta_{\hh{\GG}}(\gamma(y_i))^*(x\tens{d})\bigr)_{i\in\cI}$ is proved similarly and it follows that we get that
\[
\Delta_{\hh{\GG}}\bigl(\Linf(\hh{\HH})\bigr)\subset\M\bigl(\cK(\Ltwo(\GG))\tens\gamma(\C_0(\hh{\HH}))\bigr)
\]
and the map $\Delta_{\hh{\GG}}\colon\gamma\bigl(\Linf(\hh{\HH})\bigr)\to\M\bigl(\cK(\Ltwo(\GG))\tens\gamma(\C_0(\hh{\HH}))\bigr)$ is strict. All other conditions of Definition \ref{DefEQHS} may be checked directly and we find that $\XX$ is an embeddable quantum homogeneous space with \cst-version $\C_0(\XX)=\gamma\bigl(\C_0(\hh{\HH})\bigr)$.
\end{proof}

\section{Co-compact case}\label{ccSect}

\begin{proposition}\label{cocomp}
Let $\GG$ be a locally compact quantum group and let $\sA\subset\C_0(\GG)$ be a non-degenerate \cst-subalgebra such that
\begin{itemize}
\item[$\blacktriangleright$] $\Delta_{\GG}(\sA)\subset\M\bigl(\C_0(\GG)\tens\sA\bigr)$,
\item[$\blacktriangleright$] $\bigl[(\C_0(\GG)\tens\I)\Delta_{\GG}(\sA)\bigr]=\C_0(\GG)\tens\sA$.
\end{itemize}
Let $\sN$ be the strong closure of $\sA$. Then
\begin{enumerate}
\item\label{cocomp1} $\Delta_{\GG}(\sN)\subset\Linf(\GG)\vtens\sN$,
\item the pair $(\sN,\gamma)$ with $\gamma=\bigl.\Delta_{\GG}\bigr|_{\sN}$ is an embeddable quantum homogeneous whose \cst-version coincides with $\bigl(\sA,\bigl.\Delta_{\GG}\bigr|_{\sA}\bigr)$.
\end{enumerate}
\end{proposition}

\begin{proof}
Statement \eqref{cocomp1} follows immediately from the assumptions. It remains to show that
\begin{equation}\label{contNM}
\Delta_{\GG}(\sN)\subset\M\bigl(\cK(\Ltwo(\GG))\tens\sA\bigr)
\end{equation}
and the map
\begin{equation}\label{beStrict}
\sN\ni{x}\longmapsto\Delta_{\GG}(x)\in\M\bigl(\cK(\Ltwo(\GG))\tens\sA\bigr)
\end{equation}
is strict (cf.~Definition \ref{DefQHS}). For this consider a bounded net $(a_i)_{i\in\cI}$ of elements of $\sA$ convergent to $x\in\sN$ in the strong${}^*$ topology. For any $y\in\cK\bigl(\Ltwo(\GG)\bigr)\tens\sA$ we have
\[
\Delta_{\GG}(a_i)y=W(a_i\tens\I)W^*y.
\]
Since $W\in\M\bigl(\cK(\Ltwo(\GG))\tens\C_0(\GG)\bigr)$ we see that $\Delta_{\GG}(a_i)d=W(a_i\tens\I)z$, where $z\in\cK\bigl(\Ltwo(\GG)\bigr)\tens\C_0(\GG)$. In particular
\[
\Delta_{\GG}(a_i)y\xrightarrow[i\in\cI]{}\Delta_{\GG}(x)y
\]
in norm. Similarly $\Delta_{\GG}(a_i)^*y\xrightarrow[i\in\cI]{}\Delta_{\GG}(x)^*y$ (remember $a_i\xrightarrow[i]{}x$ in the strong${}^*$ topology). It follows that $\Delta_{\GG}(x)$ is a strict limit of elements of $\M\bigl(\C_0(\GG)\tens\sA\bigr)$ and, in particular, we get \eqref{contNM}.

Strictness of \eqref{beStrict} is established in exactly the same way: for a bounded and strong${}^*$-convergent net $(x_i)_{i\in\cI}$ of elements of $\sN$ and $y\in\cK\bigl(\Ltwo(\GG)\bigr)\tens\sA$ the nets $\bigl(\Delta_{\GG}(x_i)y\bigr)_{i\in\cI}$ and $\bigl(\Delta_{\GG}(x_i)^*y\bigr)_{i\in\cI}$ are norm-convergent to $\Delta_{\GG}\bigl(\lim\limits_{i\in\cI}x_i\bigr)y$ and $\Delta_{\GG}\bigl(\lim\limits_{i\in\cI}x_i\bigr)^*y$ respectively.
\end{proof}

In the proof of Proposition \ref{cocomp} we used the well-known fact that a bounded strongly convergent net of operators multiplied by a compact operator is norm-convergent.

\begin{definition}\label{DefCC}
Let $\GG$ be a locally compact quantum group. A quantum homogeneous space $\XX$ for $\GG$ such that $\C_0(\XX)\subset\C_0(\GG)$ is called \emph{co-compact}.
\end{definition}

\begin{example}\label{ccEx}
\noindent\begin{enumerate}
\item Let $\GG$ be a compact quantum group with faithful Haar measure. In \cite[Definition 1.8]{Pod1} Podle\'s defined an embeddable action of $\GG$ to be a ``compact'' \cst-quantum $\GG$-space $\XX$ (i.e.~the corresponding \cst-algebra $\C_0(\XX)$ is unital and hence written $\C(\XX)$) for which there exists $\Psi\in\Mor\bigl(\C(\XX),\C(\GG)\bigr)$ such that $(\id\tens\Psi)\comp\delta_\XX=\Delta_{\GG}\comp\Psi$ (in fact this definition already appeared in \cite[Definicja 2.9]{podPHD}). Letting $\Linf(\XX)$ be the strong closure of $\Psi\bigl(\C(\XX)\bigr)$ inside $\B\bigl(\Ltwo(\GG)\bigr)$ we obtain an embeddable quantum homogeneous space in the sense of Definition \ref{DefEQHS}. Moreover all such quantum homogeneous spaces are clearly co-compact.
\item\label{ccEx2} Let $\GG$ be a locally compact quantum group and let $\HH$ be a compact quantum subgroup of $\GG$ i.e.~we assume that there exists a surjective morphism $\pi\colon\C^\uu_0(\GG)\to\C^\uu_0(\HH)$ intertwining the corresponding comultiplications. Using $\pi$ we define a \cst-algebra of elements that are constant on the right $\HH$-cosets:
\[
\CC^\uu_0(\GG/\HH)=\bigl\{a\in\CC^\uu_0(\GG)\st(\id\tens\pi)(\Delta^\uu_{\GG}(a))=a\tens\I\bigr\}.
\]
Using \cite[Theorem 5.1]{Psol} (note that this theorem applies to our situation, i.e.~$\C_0^\uu(\GG)$ is a bisimplifiable Hopf \cst-algebra with a continuous counit) we see that the comultiplication $\Delta^\uu_{\GG}$ restricted to $\C^\uu_0(\GG/\HH)$ gives rise to a continuous coaction of $\bigl(\C_0^\uu(\GG),\Delta_{\GG}^\uu\bigr)$ on $\C^\uu_0(\GG/\HH)$. We shall denote it by $\delta^\uu_{\GG/\HH}$. Let us define $\C_0(\GG/\HH)$ as the image under the reducing morphism $\Lambda_\GG$:
\begin{equation}\label{GG0}
\C_0(\GG/\HH)=\Lambda_{\GG}\bigl(\C_0^\uu(\GG/\HH)\bigr).
\end{equation}
(in other words \eqref{GG0} defines the quantum space $\GG/\HH$). Note that
\[
\begin{split}
\bigl[\Delta_{\GG}\bigl(\C_0(\GG/\HH)\bigr)&\bigl(\C_0(\GG)\tens\I\bigr)\bigr]\\
&=
\bigl[(\Lambda_{\GG}\tens\Lambda_{\GG})\bigl(\Delta^\uu_{\GG/\HH}\bigl(\C^\uu_0(\GG/\HH)\bigr)
\bigl(\C^\uu_0(\GG)\tens\I\bigr)\bigr)\bigr]\\
&=\C_0(\GG)\tens\C_0(\GG/\HH).
\end{split}
\]
This shows that $\delta^\uu_{\GG/\HH}$ descends (from universal to reduced level) to a continuous action $\delta_{\GG/\HH}$ of $\GG$ on $\GG/\HH$. By Proposition
\ref{cocomp} $\GG/\HH$ becomes an embeddable quantum homogeneous space with $\Linf(\GG/\HH)$ defined as the strong closure of $\C_0(\GG/\HH)$ in $\C_0(\GG)$ and the \cst-version $\C_0(\GG/\HH)$. Actually, it is not difficult to see that it is of quotient type. By its very definition $\GG/\HH$ is co-compact.
\end{enumerate}
\end{example}

The name ``co-compact'' used in Definition \ref{DefCC} is justified by the following theorem:

\begin{theorem}\label{ccthm}
Let $\GG$ be a locally compact quantum group and let $\XX$ be an embeddable quantum homogeneous space of quotient type related to a closed quantum subgroup $\HH$ of $\GG$. Then
\begin{enumerate}
\item\label{ccthm1} if $\HH$ is compact then $\XX$ is co-compact,
\item\label{ccthm2} if there exists a non-zero $x\in\C_0(\XX)\cap\C_0(\GG)$ then $\HH$ is compact.
\end{enumerate}
\end{theorem}

\begin{proof}
\eqref{ccthm1} has been already explained in Example \ref{ccEx}\eqref{ccEx2} (assume $\XX$ is of quotient type).

For the proof of \eqref{ccthm2} recall from \cite[Section 4]{Vaes} that the quantum homogeneous space $\GG/\HH$ (of quotient type) is defined in such a way that $\C_0(\GG/\HH)$ is a subalgebra of the corresponding von Neumann algebra which we call $\Linf(\GG/\HH)$. This last algebra is the fixed point algebra $\Linf(\GG)^\alpha$, where $\alpha\colon\Linf(\GG)\to\Linf(\GG)\vtens\Linf(\HH)$ is defined by the property that
\[
(\id\tens\alpha)(W)=W_{12}V_{13},
\]
where $V\in\M\bigl(\C_0(\hh{\GG})\tens\C_0(\HH)\bigr)$ is the bicharacter corresponding to the inclusion of $\HH$ as a closed subgroup of $\GG$ (cf.~\cite[Subsection 1.3, Theorems 3.3 \& 3.6]{subgroups}). Therefore for any $x\in\C_0(\GG/\HH)$ we have $\alpha(x)=x\tens\I$. We note that it can be shown that $\alpha$ restricted to $\C_0(\GG/\HH)$ coincides with the \emph{right quantum homomorphism} corresponding to $V$ (\cite{mrw} and \cite[Subsection 1.3]{subgroups}).

So let $x$ be a non-zero element of $\C_0(\GG/\HH)$ such that also $x\in\C_0(\GG)$. Since $(y\tens\I)\alpha(x)\in\C_0(\GG)\tens\C_0(\HH)$ (by \cite[Theorems 3.5 \& 3.6(3)]{subgroups}), it follows that $(\omega\tens\id)\alpha(x)\in\C_0(\HH)$ for any $\omega\in\C_0(\GG)^*$. In particular for $\omega$ such that $\omega(x)=1$ we get $\I\in\C_0(\HH)$, i.e.~$\HH$ is compact.
\end{proof}

\begin{remark}
The proof of Theorem \ref{ccthm}\eqref{ccthm1} shows that if $\XX$ is an embeddable quantum homogeneous space of quotient type such that $\C_0(\XX)\cap\C_0(\GG)\neq\{0\}$ then $\C_0(\XX)\subset\C_0(\GG)$ and $\XX$ is co-compact.
\end{remark}

\section{The non-commutative analog of the quotient by the diagonal subgroup}\label{diagSect}

Let $G$ be a locally compact group. In this section we will study a non-commutative analog of the homogeneous space obtained as a quotient of $G\times{G}$ by the diagonal subgroup. For technical reasons we will work with $G\times{G^\op}$ instead of $G\times{G}$. In this case the diagonal subgroup is embedded via
\[
G\ni{g}\longmapsto(g,g^{-1})\in{G\times{G^\op}}.
\]
The resulting homogeneous space consist of the classes of the form
\[
\bigl[(g_1,g_2)\bigr]=\bigl\{(g_1g,g^{-1}g_2)\st{g\in{G}}\bigr\},\qquad\qquad(g_1,g_2\in{G})
\]
and the space $X$ of classes is homeomorphic to $G$. Indeed, the homeomorphism is provided by mapping a class $\bigl[(g_1,g_2)\bigr]$ to the product $g_1g_2\in{G}$. On the level of algebras of functions we get a natural identification of  $\Delta_G\bigl(\Linf(G)\bigr)\subset\Linf(G)\vtens\Linf(G)$ with the algebra of functions on the homogeneous space $(G\times{G^\op})/G$.

Let now $\GG$ be a locally compact quantum group. It is well known that the analog of the direct product $\GG\times\GG^\op$ of $\GG$ with its opposite quantum group can be defined by putting
\[
\C_0(\GG\times\GG^\op)=\C_0(\GG)\tens\C_0(\GG^\op)
\]
and
\[
\Delta_{\GG\times\GG^\op}=(\id\tens\sigma\tens\id)\comp(\Delta_{\GG}\tens\Delta_{\GG^\op}),
\]
where $\sigma$ is the flip morphism $\C_0(\GG)\tens\C_0(\GG^\op)\to\C_0(\GG^\op)\tens\C_0(\GG)$.

In what follows we shall prove that setting $\Linf(\XX)=\Delta_{\GG}\bigl(\Linf(\GG)\bigr)\subset\Linf(\GG)\vtens\Linf(\GG)$ we define an embeddable quantum homogeneous $\GG\times\GG^\op$-space. This space is the quantum analog of the homogeneous space obtained (in the classical case) by taking the quotient of $\GG\times\GG^\op$ by the diagonal subgroup. We will show in Proposition \ref{GGprop} that $\XX$ is of quotient type if and only if $\GG$ is a classical group. In the course of the proof we will make extensive use of the concept of the \emph{opposite quantum group} $\GG^\op$ (mentioned above) and the \emph{commutant quantum group} $\GG'$. For details we refer the reader to \cite[Section 4]{KVvN}. In particular we will need the facts that $\hh{\GG^\op}=\hh{\GG}'$, $\hh{\GG'}=\hh{\GG}^\op$ and their extended version: $\hh{\GG\times\GG^\op}=\hh{\GG}\times\hh{\GG}'$ (see \cite{KVvN}).

\begin{proposition}\label{propXG}
Let $\GG$ be a locally compact quantum group and let $\Linf(\XX)$ be the image of $\Linf(\GG)$ under $\Delta_\GG$ considered as a map from $\Linf(\GG)$ to $\Linf(\GG)\vtens\Linf(\GG^\op)$. Then:
\begin{enumerate}
\item $\Delta_{\GG\times\GG^\op}\bigl(\Linf(\XX)\bigr)\subset\Linf(\GG\times\GG^\op)\vtens\Linf(\XX)$,
\item $\XX$ is an embeddable quantum homogeneous space with \cst-version $\C_0(\XX)=\Delta_{\GG}\bigl(\C_0(\GG)\bigr)$.
\end{enumerate}
\end{proposition}

\begin{proof}
Let
\begin{equation}\label{DelXX}
\C_0(\XX)=\Delta_{\GG}\bigl(\C_0(\GG)\bigr)\subset\M\bigl(\C_0(\GG\times\GG^\op)\bigr)=\M\bigl(\C_0(\GG)\tens\C_0(\GG^\op)\bigr)
\end{equation}
(it is clear that $\Linf(\XX)$ as defined in the statement of the theorem is the strong closure of so defined $\C_0(\XX)$). We will first show that
\[
\Bigl[\bigl(\C_0(\GG\times\GG^\op)\tens\I_{\C_0(\XX)}\bigr)\Delta_{\GG\times\GG^\op}\bigl(\C_0(\XX)\bigr)\Bigr]=
\C_0(\GG\times\GG^\op)\tens\C_0(\XX)
\]
Indeed
\[
 \begin{split}
  \Bigl[\bigl(\C_0(\GG&\times\GG^\op)\tens\I_{\C_0(\XX)}\bigr)\Delta_{\GG\times\GG^\op}\bigl(\C_0(\XX)\bigr)\Bigr]\\
  &=\Bigl[\bigl(\C_0(\GG\times\GG^\op)\tens\I_{\C_0(\GG)}\tens\I_{\C_0(\GG^\op)}\bigr)\Delta_{\GG\times\GG^\op}
  \bigl(\Delta_{\GG}(\C_0(\GG))\bigr)\Bigr]\\
  &=\Bigl[\bigl(\C_0(\GG)\tens\C_0(\GG^\op)\tens\I\tens\I\bigr)
  \bigl((\id\tens\sigma\tens\id)(\Delta_{\GG}\tens\Delta_{\GG^\op})\Delta_{\GG}(\C_0(\GG))\bigr)\Bigr]\\
  &=\Bigl[\bigl(\C_0(\GG)\tens\C_0(\GG^\op)\tens\I\tens\I\bigr)
  \bigl((\id\tens\sigma\tens\id)(\id\tens\id\tens\Delta_{\GG^\op})(\id\tens\Delta_{\GG})\Delta_{\GG}(\C_0(\GG))\bigr)\Bigr]\\
  &=\C_0(\GG)\tens\Bigl[
  \bigl(\C_0(\GG^\op)\tens\I\tens\I\bigr)\bigl((\sigma\tens\id)(\id\tens\Delta_{\GG^\op})\Delta_{\GG}(\C_0(\GG))\bigr)\Bigr]\\
  &=\C_0(\GG)\tens\Bigl[
  \bigl(\C_0(\GG^\op)\tens\I\tens\I\bigr)\bigl((\sigma\tens\id)(\id\tens\sigma)(\Delta_{\GG}\tens\id)\Delta_{\GG}(\C_0(\GG))\bigr)\Bigr]\\
  &=\C_0(\GG)\tens\Bigl[
  \bigl(\C_0(\GG^\op)\tens\I\tens\I\bigr)\bigl((\id\tens\Delta_{\GG})\Delta_{\GG^\op}(\C_0(\GG))\bigr)\Bigr]\\
  &=\C_0(\GG)\tens\C_0(\GG^\op)\tens\Delta_{\GG}\bigl(\C_0(\GG)\bigr)\\
  &=\C_0(\GG\times\GG^\op)\tens\C_0(\XX).
 \end{split}
\]
In the fourth equality above we used the fact that
\[
\bigl[\bigl(\C_0(\GG)\tens\I\bigr)\Delta_{\GG}\bigl(\C_0(\GG)\bigr)\bigr]=\C_0(\GG)\tens\C_0(\GG)
\]
while in the seventh one we used
\[
\bigl[\bigl(\C_0(\GG^\op)\tens\I\bigr)\Delta_{\GG^\op}\bigl(\C_0(\GG)\bigr)\bigr]=\C_0(\GG^\op)\tens\C_0(\GG).
\]
Thus $\XX$ defined via \eqref{DelXX} is a \cst-quantum $\GG\times\GG^\op$-space. Moreover, taking strong closure we put on $\XX$ the structure of an ergodic \wst-quantum $\GG\times\GG^\op$-space (which is moreover embeddable). We will show that this way $\XX$ becomes an embeddable quantum homogeneous space. To this end we need to establish that
\[
\Delta_{\GG\times\GG^\op}\bigl(\Linf(\XX)\bigr)\subset
\M\bigl(\cK(\Ltwo(\GG\times\GG^\op))\tens\C_0(\XX)\bigr)
\]
and the map $\bigl.\Delta_{\GG\times\GG^\op}\bigr|_{\Linf(\XX)}\colon\Linf(\XX)\to\M\bigl(\cK\bigl(\Ltwo(\GG\times\GG^\op)\bigr)\tens\C_0(\XX)\bigr)$ is strict. Let us denote this map by $\delta_\XX$.

Let us first observe that
\begin{equation}\label{sigma12}
 \begin{split}
  \Delta_{\GG\times\GG^\op}\comp\Delta_{\GG}
  &=(\id\tens\sigma\tens\id)\comp(\id\tens\id\tens\sigma)\comp(\Delta_{\GG}\tens\Delta_{\GG})\comp\Delta_{\GG}\\
  &=(\id\tens\sigma\tens\id)\comp(\id\tens\id\tens\sigma)\comp(\Delta_{\GG}\tens\id\tens\id)
  \comp(\Delta_{\GG}\tens\id)\comp\Delta_{\GG}\\
  &=(\sigma\tens\id\tens\id)\comp(\id\tens\Delta_{\GG}\tens\id)\comp(\id\tens\Delta_{\GG})\comp\sigma\comp\Delta_{\GG}\\
  &=(\sigma\tens\id\tens\id)\comp(\id\tens\Delta_{\GG}\tens\id)\comp(\id\tens\Delta_{\GG})\comp\Delta_{\GG^\op}\\
 \end{split}
\end{equation}
Take now $k_1\tens{k_2}\tens{\Delta_{\GG}(y)}\in\cK\bigl(\Ltwo(\GG)\bigr)\tens\cK\bigl(\Ltwo(\GG^\op)\bigr)\tens\C_0(\XX)$ and let $(y_i)_{i\in\cI}=\bigl(\Delta_{\GG}(x_i)\bigr)_{i\in\cI}$ be a norm-bounded strongly convergent net with $x_i\in\Linf(\GG)$. By \eqref{sigma12} we have
\[
 \begin{split}
  \delta_\XX(y_i)&=(\sigma\tens\id\tens\id)(\id\tens\id\tens\Delta_{\GG})(\id\tens\Delta_{\GG})\Delta_{\GG^\op}(x_i)\\
  &=\Sigma_{12}W_{34}W_{23}\Sigma_{12}W_{12}(x_i\tens\I\tens\I\tens\I)W_{12}^*\Sigma_{12}W_{23}^*W_{34}^*\Sigma_{12}.
 \end{split}
\]
It follows that
\[
 \begin{split}
  \delta_\XX(y_i)&\cdot\bigl(k_1\tens{k_2}\tens\Delta_{\GG}(y)\bigr)=\\
  &=\Sigma_{12}W_{34}W_{23}\Sigma_{12}W_{12}(x_i\tens\I\tens\I\tens\I)W_{12}^*\Sigma_{12}W_{23}^*W_{34}^*\Sigma_{12}
  \cdot{W_{34}(k_1\tens{k_2}\tens{y}\tens\I)W_{34}^*}\\
  &=\Sigma_{12}W_{34}W_{23}\Sigma_{12}W_{12}(x_i\tens\I\tens\I\tens\I)W_{12}^*\Sigma_{12}W_{23}^*
  (k_2\tens{k_1}\tens{y}\tens\I)W_{34}^*\Sigma_{12}\\
  &=(\sigma\tens\id\tens\id)(\id\tens\id\tens\Delta_{\GG})\Bigl(
  W_{23}\bigl(\Delta_{\GG^\op}(x_i)\tens\I\bigr)W_{23}^*(k_2\tens{k_1}\tens{y})\Bigr)\\
  &=(\sigma\tens\Delta_{\GG})\Bigl(
  W_{23}\bigl(\Delta_{\GG^\op}(x_i)\tens\I\bigr)W_{23}^*(k_2\tens{k_1}\tens{y})\Bigr).
 \end{split}
\]
Consider now the net
\begin{equation}\label{nnet}
 \Bigl(\bigl(\Delta_{\GG^\op}(x_i)\tens\I\bigr)W_{23}^*(k_2\tens{k_1}\tens{y})\Bigr)_{i\in\cI}.
\end{equation}
We claim that it is norm-convergent. Indeed, since $W\in\M\bigl(\cK(\Ltwo(\GG))\tens\C_0(\GG)\bigr)$, we have $W_{23}^*(k_2\tens{k_1}\tens{y})\in\cK\bigl(\Ltwo(\GG\times\GG^\op)\bigr)\tens\C_0(\GG)$, while the net
\[
\bigl(\Delta_{\GG^\op}(x_i)\tens\I\bigr)_{i\in\cI}
=\bigl((\sigma\tens\id)(y_i\tens\I)\bigr)_{i\in\cI}.
\]
is bounded and strongly convergent. It follows that \eqref{nnet} is norm-convergent and since the norm-convergence of the net $\delta_\XX(y_i)^*\cdot\bigl(k_1\tens{k_2}\tens\Delta_{\GG}(y)\bigr)$ is proved similarly. Thus we obtain strictness of $\delta_\XX$.
\end{proof}

In order to find the co-dual of $\XX$ we prove the following lemma:

\begin{lemma}\label{Dely}
Let $\GG$ be a locally compact quantum group and $W\in\Linf(\hh{\GG})\vtens\Linf(\GG)$ its Kac-Takesaki operator. Let $x\in\Linf(\GG)\vtens\Linf(\GG)$ and $y\in\B\bigl(\Ltwo(\GG)\bigr)$ be such that $W(y\tens\I)W^*=x$. Then $y\in\Linf(\GG)$ and $x=\Delta_\GG(y)$
\end{lemma}

\begin{proof}
It is enough to show that there exists $z\in\Linf(\GG)$ such that $\Delta_{\GG}(z)=x$. Indeed, our assumptions then imply that $W(y\tens\I)W^*=x=W(z\tens\I)W^*$ and it follows that $y=z\in\Linf(\GG)$. In order to prove the existence of such a $z$ it is enough to show that $(\Delta_{\GG}\tens\id)(x)=(\id\tens\Delta_{\GG})(x)$ (by \cite[Theorem 2.4.7]{VaesPHD} applied to the action of $\GG$ on itself). We compute
\[
\begin{split}
(\Delta_{\GG}\tens\id)(x)&=W_{12}x_{13}W_{12}^*\\
&=W_{12}W_{13}(y\tens\I\tens\I)W_{13}^*W_{12}^*\\
&=W_{23}W_{12}W_{23}^*(y\tens\I\tens\I)W_{23}W_{12}^*W_{23}^*\\
&=W_{23}W_{12}(y\tens\I\tens\I)W_{12}^*W_{23}^*.
\end{split}
\]
On the other hand
\[
(\id\tens\Delta_{\GG})(x)=W_{23}x_{12}W_{23}^*=W_{23}W_{12}(y\tens\I\tens\I)W_{12}^*W_{23}^*.
\]
which ends the proof.
\end{proof}

\begin{proposition}\label{GGprop}
$\XX$ be a embeddable \wst-quantum $\GG$ space space defined above. Then
\begin{enumerate}
\item\label{GGprop1} we have $\Linf(\dd{\XX})
=(J\tens\hh{J})\Delta_{\hh{\GG}^\op}(\Linf(\hh{\GG}))(J\tens\hh{J})\subset\Linf(\hh{\GG})\vtens\Linf(\hh{\GG})'$.

\item\label{GGprop2} $\XX$ is of quotient type if and only if $\GG$ is classical.
\end{enumerate}
\end{proposition}

\begin{proof}
Ad \eqref{GGprop1}. Let us recall that $x\in\dd\XX$ if $x\in\Linf(\hh{\GG\times\GG^\op})=\Linf(\hh{\GG})\vtens\Linf(\hh{\GG})'$ and for any $y\in\Linf(\GG)$ we have $x\Delta_{\GG}(y)=\Delta_{\GG}(y)x$. This means that
\[
W^*xW(y\tens\I)=(y\tens\I)W^*xW
\]
for all $y\in\Linf(\hh{\GG})$. Since $W^*xW\in\Linf(\hh{\GG})\vtens\B\bigl(\Ltwo(\GG)\bigr)$ and $\Linf(\hh{\GG})\cap\Linf(\GG)'=\Linf(\hh{\GG}^\op)\cap\Linf(\hh{\hh{\GG}^\op})=\CC\I$ (see \cite[Page 131]{VaesPHD}) we see that that $W^*xW=\I\tens{z}$ for some $z\in\B\bigl(\Ltwo(\GG)\bigr)$.

Applying Lemma \ref{Dely} with $\hh{\GG}'$ instead of $\GG$ we obtain $z\in\Linf(\hh{\GG}')=\Linf(\hh{\GG})'$. Let $\widetilde{z}=\hh{J}z\hh{J}$. We see that
\[
x=W(J\tens\hh{J})(\I\tens\widetilde{z}\:\!)(J\tens\hh{J})W^*=(J\tens\hh{J})W^*(\I\tens\widetilde{z})W(J\tens\hh{J})
=(J\tens\hh{J})\Delta_{\hh{\GG}^\op}(\widetilde{z})(J\tens\hh{J}),
\]
where we used the identity $(J\tens\hh J)W(J\tens\hh J)=W^*$ (\cite[Theorem 5.11]{mnw}).

Ad \eqref{GGprop2}.
Let us first define a normal order two automorphism $\alpha\colon\B\bigl(\Ltwo(\GG)\bigr)\to\B\bigl(\Ltwo(\GG)\bigr)$ putting
\[
\alpha(t)=\hh{J}JtJ\hh{J}
\]
(the fact that $\alpha^2=\id$ follows immediately from \cite[Corollary 1.12]{KVvN}). We have the following relations (cf.~\cite[Remarks before Proposition 4.2]{KVvN}):
\begin{equation}\label{relalpha}
\begin{aligned}
\alpha\bigl(\Linf(\hh{\GG}^\op)\bigr)&=\Linf(\hh{\GG}'),&
\quad\Delta_{\hh{\GG}'}\comp\alpha&=(\alpha\tens\alpha)\comp\Delta_{\hh{\GG}^\op},\\
\alpha\bigl(\Linf(\GG^\op)\bigr)&=\Linf(\GG'),&
\quad\Delta_{\GG'}\comp\alpha&=(\alpha\tens\alpha)\comp\Delta_{\GG^\op}.
\end{aligned}
\end{equation}
To show the first one we compute remembering that $J$ implements $\hh{R}$:
\[
\begin{split}
\Delta_{\hh{\GG}'}\bigl(\alpha(x)\bigr)&=\Delta_{\hh{\GG^\op}}\bigl(\alpha(x)\bigr)\\
&=\Delta_{\hh{\GG^\op}}(\hh{J}JxJ\hh{J})\\
&=(\hh{J}\tens\hh{J})\Delta_{\hh{\GG}}(JxJ)(\hh{J}\tens\hh{J})\\
&=(\hh{J}J\tens\hh{J}J)\Delta_{\hh{\GG}^\op}(x)(J\hh{J}\tens{J}\hh{J})\\
&=(\alpha\tens\alpha)\bigl(\Delta_{\hh{\GG}^\op}(x)\bigr),
\end{split}
\]
where in the third equality we used the fact that
\[
\Delta_{\hh{\GG^\op}}(z)=(\hh{J}\tens\hh{J})\Delta_{\hh{\GG}}(\hh{J}z\hh{J})(\hh{J}\tens\hh{J})
\]
for all $z\in\Linf(\hh{\GG})'$. The other relation \eqref{relalpha} can be derived from the first one using duality.

Repeating the construction of $\XX$ (Proposition \ref{propXG}) for the quantum group $\hh{\GG}$ we obtain an embeddable quantum homogeneous space $\YY$. This means that
\[
\Linf(\YY)=\Delta_{\hh{\GG}}\bigl(\Linf(\hh{\GG})\bigr)\subset\Linf(\hh{\GG}\times\hh{\GG}^\op)\quad\text{and}
\quad\delta_{\YY}=\bigl.\Delta_{\hh{\GG}\times\hh{\GG}^\op}\bigr|_{\Linf(\YY)}.
\]
Using $\alpha$ we can define new embeddable quantum homogeneous spaces:
\begin{itemize}
\item[$\blacktriangleright$] $\XX_\alpha$ with $\Linf(\XX_\alpha)=(\id\tens\alpha)\bigl(\Linf(\XX)\bigr)\subset\Linf(\GG\times\GG')$,
\item[$\blacktriangleright$] $\YY_\alpha$ with $\Linf(\YY_\alpha)=(\id\tens\alpha)\bigl(\Linf(\YY)\bigr)\subset\Linf(\hh{\GG}\times\hh{\GG}')$,
\end{itemize}
We have
\begin{equation}\label{delalpha}
\begin{split}
\delta_{\XX_\alpha}\comp(\id\tens\alpha)&=\bigl((\id\tens\alpha)\tens(\id\tens\alpha)\bigr)\comp\delta_\XX,\\
\delta_{\YY_\alpha}\comp(\id\tens\alpha)&=\bigl((\id\tens\alpha)\tens(\id\tens\alpha)\bigr)\comp\delta_\YY.
\end{split}
\end{equation}
As before we only show the first relation of \eqref{delalpha}:
\[
\begin{split}
\Delta_{\hh{\GG}\times\hh{\GG}'}\bigl(\Linf(\dd{\XX})\bigr)
&=\Delta_{\hh{\GG}\times\hh{\GG}'}\Bigl((\id\tens\alpha)\bigl(\Delta_{\hh{\GG}}(\Linf(\hh{\GG}))\bigr)\Bigr)\\
&=\bigl((\id\tens\sigma\tens\id)\comp(\Delta_{\hh{\GG}}\tens\Delta_{\hh{\GG}'})\comp(\id\tens\alpha)
\comp\Delta_{\hh{\GG}}\bigr)\bigl(\Linf(\hh{\GG})\bigr)\\
&=\bigl((\id\tens\sigma\tens\id)\comp(\id\tens\id\tens\alpha\tens\alpha)\comp(\Delta_{\hh{\GG}}\tens\Delta_{\hh{\GG}^\op})
\comp\Delta_{\hh{\GG}}\bigr)\bigl(\Linf(\hh{\GG})\bigr)\\
&=\bigl((\id\tens\alpha\tens\id\tens\alpha)\comp(\id\tens\sigma\tens\id)\comp(\Delta_{\hh{\GG}}\tens\Delta_{\hh{\GG}^\op})
\comp\Delta_{\hh{\GG}}\bigr)\bigl(\Linf(\hh{\GG})\bigr)\\
&=\bigl((\id\tens\alpha\tens\id\tens\alpha)\comp\Delta_{\hh{\GG}\times\hh{\GG}^\op}\comp
\Delta_{\hh{\GG}}\bigr)\bigl(\Linf(\hh{\GG})\bigr),
\end{split}
\]
where we used the first line of \eqref{relalpha} in the third equality.

Now statement \eqref{GGprop1} of this proposition means that $\dd{\XX}=\YY_\alpha$. Similarly we have $\dd{\YY}=\XX_\alpha$. Moreover $\XX$ is of quotient type if and only if $\XX_\alpha$ is of quotient type. This is clear from Lemma \ref{charQT}: $\XX$ is of quotient type if and only if $\dd{\XX}$ has the property that $\Linf(\dd{\XX})=\gamma\bigl(\Linf(\hh{\HH})\bigr)$, where $\HH$ is a closed quantum subgroup of $\GG$ in the sense of Vaes and $\gamma\colon\Linf(\hh{\HH})\to\Linf(\hh{\GG\times\GG^\op})=\Linf(\hh{\GG}\times\hh{\GG}')$ is the corresponding normal inclusion (cf.~\cite{Vaes,subgroups}). Clearly $(\id\tens\alpha)\comp\gamma$ provides a normal inclusion making $\HH$ a closed quantum subgroup of $\GG\times\GG'$. Since the resulting embeddable quantum homogeneous space is $\YY$ we see that $\XX_\alpha$ is of quotient type.

We will now assume that $\XX_\alpha$ is of quotient type and show that $\GG$ must then be a classical group. It follows easily from this assumption that $\delta_\YY\bigl(\Linf(\YY)\bigr)\subset\Linf(\YY)\vtens\Linf(\YY)$. This means that for any $x\in\Linf(\hh{\GG})$ there exists $y\in\Linf(\hh{\GG})\vtens\Linf(\hh{\GG})$ such that
\begin{equation}\label{xy}
\Delta_{\hh{\GG}\times\hh{\GG}^\op}\bigl(\Delta_{\hh{\GG}}(x)\bigr)=(\Delta_{\hh{\GG}}\tens\Delta_{\hh{\GG}})(y).
\end{equation}
Using \eqref{sigma12} applied to $\hh{\GG}$ we compute
\[
\begin{split}
\Delta_{\hh{\GG}\times\hh{\GG}^\op}\comp\Delta_{\hh{\GG}}(x)
&=(\sigma\tens\id\tens\id)\comp(\id\tens\Delta_{\hh{\GG}}\tens\id)\comp(\id\tens\Delta_{\hh{\GG}})\comp\Delta_{\hh{\GG}^\op}\\
&=(\sigma\tens\id\tens\id)\comp(\id\tens\id\tens\Delta_{\hh{\GG}})\comp(\id\tens\Delta_{\hh{\GG}})\comp\Delta_{\hh{\GG}^\op}\\
&=(\id\tens\id\tens\Delta_{\hh{\GG}})\comp(\sigma\tens\id\tens\id)\comp(\id\tens\Delta_{\hh{\GG}})\comp\Delta_{\hh{\GG}^\op}\\
&=(\id\tens\id\tens\Delta_{\hh{\GG}})\comp(\id\tens\sigma)\comp(\Delta_{\hh{\GG}}\tens\id)\comp\Delta_{\hh{\GG}}.
\end{split}
\]
Thus \eqref{xy} reads as
\[
\begin{split}
\bigl((\id\tens\id\tens\Delta_{\hh{\GG}})\comp(\id\tens\sigma)\comp(&\Delta_{\hh{\GG}}\tens\id)\comp\Delta_{\hh{\GG}}\bigr)(x)\\
&=(\Delta_{\hh{\GG}}\tens\Delta_{\hh{\GG}})(y)=
\bigl((\id\tens\id\tens\Delta_{\hh{\GG}})\comp(\Delta_{\hh{\GG}}\tens\id)\bigr)(y)
\end{split}
\]
It follows that for any $x\in\Linf(\hh{\GG})$ there exists $y\in\Linf(\hh{\GG})\vtens\Linf(\hh{\GG})$ such that
\begin{equation}
\label{duzoDelt}
\bigl((\id\tens\sigma)\comp(\Delta_{\hh{\GG}}\tens\id)\comp\Delta_{\hh{\GG}}\bigr)(x)=(\Delta_{\hh{\GG}}\tens\id)(y).
\end{equation}
Now we note that since for any $y\in\Linf(\hh{\GG})\vtens\Linf(\hh{\GG})$ we have
\[
(\Delta_{\hh{\GG}}\tens\id\tens\id)\bigl((\Delta_{\hh{\GG}}\tens\id)(y)\bigr)=
(\id\tens\Delta_{\hh{\GG}}\tens\id)\bigl((\Delta_{\hh{\GG}}\tens\id)(y)\bigr),
\]
we can apply $\Delta_{\hh{\GG}}\tens\id\tens\id$ and $\id\tens\Delta_{\hh{\GG}}\tens\id$ to the left hand side of \eqref{duzoDelt} and get equal results. In other words, for any $x\in\Linf(\hh{\GG})$ we have
\[
\bigl((\Delta_{\hh{\GG}}\tens\id\tens\id)\comp(\id\tens\sigma)\comp(\Delta_{\hh{\GG}}\tens\id)\comp\Delta_{\hh{\GG}}\bigr)(x)
=\bigl((\id\tens\Delta_{\hh{\GG}}\tens\id)\comp(\id\tens\sigma)\comp(\Delta_{\hh{\GG}}\tens\id)\comp\Delta_{\hh{\GG}}\bigr)(x),
\]
i.e.
\begin{equation}\label{idhgcom}
(\Delta_{\hh{\GG}}\tens\id\tens\id)\comp(\id\tens\sigma)\comp(\Delta_{\hh{\GG}}\tens\id)\comp\Delta_{\hh{\GG}}
=(\id\tens\Delta_{\hh{\GG}}\tens\id)\comp(\id\tens\sigma)\comp(\Delta_{\hh{\GG}}\tens\id)\comp\Delta_{\hh{\GG}}
\end{equation}

Let $\hh{W}\in\Linf(\GG)\vtens\Linf(\hh{\GG})$ be the Kac-Takesaki operator for $\hh{\GG}$. Applying \eqref{idhgcom} to the second leg of $\hh{W}$ and remembering that $(\id\tens\Delta_{\hh{\GG}})(\hh{W})=\hh{W}_{12}\hh{W}_{13}$ we obtain
\begin{equation}\label{idhgcom1}
\hh{W}_{12}\hh{W}_{13}\hh{W}_{15}\hh{W}_{14}=\hh{W}_{12}\hh{W}_{15}\hh{W}_{13}\hh{W}_{14}.
\end{equation}
Since the first leg of $\hh{W}$ generates $\Linf(\GG)$, slicing appropriately equation \eqref{idhgcom1} we conclude that $abdc=adbc$ for any $a,b,c,d\in\Linf(\GG)$. Putting $a=c=\I$ we get $bd=db$ for any $d,b\in\Linf(\GG)$. Commutativity of $\Linf(\GG)$ is equivalent to $\GG$ being a classical group.
\end{proof}

\section{Strongly embeddable quantum homogeneous spaces}\label{strongSect}

Consider a classical locally compact group $G$ and its homogeneous space $X$. The space $X$ is equivariantly homeomorphic to the quotient space $G/H$, where $H$ is a closed subgroup of $G$. In particular we have a surjective continuous (and open) map $p\colon{G}\to{X}$ whose every fiber is homeomorphic to $H$. In other words we obtain a principal bundle $G\xrightarrow{p}{X}$ with structure group $H$. We say that the bundle $G\xrightarrow{p}{X}$ is \emph{trivial} if $G$ is homeomorphic to the Cartesian product $H\times{X}$ and this homeomorphism is a morphism of bundles over $X$ (classical theory of principal bundles is contained e.g.~in \cite[Chapter 4]{husemoller}). The triviality of $G\xrightarrow{p}{X}$ is equivalent to the existence of a continuous map $s\colon{X}\to{G}$ (called a \emph{cross-section} of the bundle) such that $p\comp{s}=\id$ (\cite[Corollary 8.3]{husemoller}).

In this section we consider certain conditions on an embeddable quantum homogeneous space which are analogs of the condition of triviality of the principal bundle discussed above. In the next definition we consider a morphism from a \cst-algebra $\sA$ to itself and assume it to be idempotent. This means that the canonical extension of the morphism to a map $\M(\sA)\to\M(\sA)$ is idempotent.

\begin{definition}\label{projem}
Let $\GG$ be a locally compact quantum group and $\XX$ an embeddable quantum homogeneous space. We say that $\XX$ is projectively embeddable if there exists a $\pi\in\Mor\bigl(\C_0(\GG),\C_0(\GG)\bigr)$ such that $\pi^2=\pi$, $\Delta_\GG\comp\pi=(\id\tens\pi)\comp\Delta_\GG$ and $\C_0(\XX)=\pi\bigl(\C_0(\GG)\bigr)$.
\end{definition}

It can be shown that, if $G$ is a classical locally compact group, then any projectively embeddable (quantum) homogeneous space for $G$ is a classical homogeneous space $X$ such that the principal bundle $G\to{X}$ is trivial. However, the condition of projective embeddability seems too strong for the non-commutative setting. It will be shown is Subsection \ref{azbsubs} that the quantum deformation of the action of the ``$az+b$'' group on its homogeneous space $\CC$ (which should certainly be considered as giving rise to a trivial principal bundle) is not projectively embeddable.

Before stating the condition replacing projective embeddability let us note the following proposition:

\begin{proposition}\label{se} Let $\GG$ be a locally compact quantum group and let $\sD\subset\M\bigl(\C_0(\GG)\bigr)$ be a \cst-subalgebra such that $\bigl[\Delta_{\GG}(\sD)(\C_0(\GG)\tens\I)\bigr]=\C_0(\GG)\tens\sD$. Let $U\in\M\bigl(\cK(\Ltwo(\GG))\tens\sD\bigr)$ be a unitary such that $\Delta_{\GG}(d)=U(d\tens\I)U^*$ for all $d\in\sD$ and let $\sN\subset\Linf(\GG)$ be the strong closure of $\sD$. Then defining $\XX$ so that $\Linf(\XX)=\sN$ and $\delta_\XX=\bigl.\Delta_\GG\bigr|_{\Linf(\XX)}$ we obtain an embeddable quantum homogeneous space $\XX$ and, moreover, $\C_0(\XX)=\sD$.
\end{proposition}

\begin{proof}
Let $(x_i)_{i\in\cI}$ be a bounded strong${}^*$-convergent net of elements of $\sN$ and take $y\in\cK\bigl(\Ltwo(\GG)\bigr)\tens\sD$. Since $Uy\in\cK\bigl(\Ltwo(\GG)\bigr)\tens\sD$, we see that the net $\bigl((x_i\tens\I)Uy\bigr)_{i\in\cI}$ is norm-convergent. We assumed that $\Delta_{\GG}(x_i)y=U(x_i\tens\I)U^*y$ for all $i$, and so it follows that the net $\bigl(\Delta_{\GG}(x_i)y\bigr)_{i\in\cI}$ is norm-convergent. Let now $y\in\sN$ and $(d_i)$ be a norm-bounded net of elements of $\sD$ converging in the strong${}^*$ topology to $y$. The above reasoning shows that
\begin{itemize}
\item[$\blacktriangleright$] $\Delta_\GG(y)$ is a strict limit of $\Delta_\GG(d_i)\in\M\bigl(\cK(\Ltwo(\GG)\tens\sD\bigr)$ and $\Delta_\GG(y)\in\M\bigl(\cK(\Ltwo(\GG))\tens\sD\bigr)$,
\item[$\blacktriangleright$] the map $\sN\ni{y}\mapsto\Delta_{\GG}(y)\in\M\bigl(\cK(\Ltwo(\GG))\tens\sD\bigr)$ is strict (cf.~\cite[Remark after Definition 3.1]{Vaes}).
\end{itemize}
Checking the remaining conditions of Definition \ref{DefQHS} is easy.
\end{proof}

\begin{definition}
Let $\XX$ be an embeddable quantum homogeneous space for a locally compact quantum group $\GG$. We say that $\XX$ is a \emph{strongly embeddable} quantum homogeneous space if there exists a unitary $U\in\M\bigl(\cK(\Ltwo(\GG))\tens\C_0(\XX)\bigr)$ such that $\delta_{\XX}(x)=U(x\tens\I)U^*$ for all $x\in\Linf(\XX)$.
\end{definition}

\begin{example}
Let $\GG$ be a locally compact quantum group and $\pi\in\Mor\bigl(\C_0(\GG),\C_0(\GG)\bigr)$ a morphism satisfying $\pi^2=\pi$ and $\Delta_{\GG}\comp\pi=(\id\tens\pi)\comp\Delta_{\GG}$. Then $\sD=\pi\bigl(\C_0(\GG)\bigr)$ and $U=(\id\tens\pi)W\in\M\bigl(\cK(\Ltwo(\GG))\tens\sD\bigr)$ satisfies the conditions of Proposition \ref{se}. Thus defining $\XX$ by setting $\Linf(\XX)=\sD''$ we obtain a strongly embeddable quantum homogeneous space $\XX$. In particular any projectively embeddable quantum homogeneous space for $\GG$ is strongly embeddable
\end{example}

The condition of strong embeddability appear to be a satisfactory analog of the condition that the principal bundle corresponding to the action is trivial. In the following subsections we discuss examples of strongly embeddable quantum homogeneous spaces.

\subsection{Bicrossed products of locally compact groups}\label{expbc}

In this subsection we shall study examples of quantum homogeneous spaces of quotient type provided by the bicrossed product construction. Its interesting aspect is related to the generic non-regularity of quantum groups arising in this construction (\cite{BSV}, recall that the general result on the  existence of the \cst-version of a given quotient type \wst-quantum $\GG$-space is based on the regularity of $\GG$). Our presentation of the bicrossed product construction closely follows \cite{BSV}. For earlier approaches see also \cite{Mbicross,VaesB}.

Let $G_1$, $G_2$ be a pair of closed subgroups of a locally compact group $G$ such that $G_1\cap{G_2}=\{e\}$ and the complement of $G_1G_2$ is of measure $0$. Let $\Linf(\hh{G}_1)$ denote the von Neumann algebra generated by right shifts on $\Ltwo(G_1)$. Thus $\Linf(\hh{G}_1)'$ is the von Neumann algebra generated by the left shifts. The same notation will be used with $G_1$ replaced by $G_2$.

Let $p_1\colon{G}\to{G_1}$ be the measurable map (defined almost everywhere) given by $p_1(g_1g_2)=g_1$. Similarly we have the map $p_2\colon{G}\to{G_2}$ such that $p_2(g_1g_2)=g_2$. Using $p_1$ and $p_2$ we introduce the pair of coactions
\[
\begin{split}
\alpha\colon\Linf(G_1)&\to\Linf(G_2)\vtens\Linf(G_1),\qquad\alpha(x)(g_2,g_1)=x\bigl(p_1(g_2g_1)\bigr),\\
\beta\colon\Linf(G_2)&\to\Linf(G_2)\vtens\Linf(G_1),\qquad\beta(y)(g_2,g_1)=y\bigl(p_2(g_2g_1)\bigr).
\end{split}
\]
We denote by $\GG$ the quantum group obtained from the above data via the bicrossed product construction. The von Neumann algebra $\Linf(\GG)$ is given by the crossed product
\begin{equation}\label{bcp}
\Linf(\GG)=G_2\ltimes_\alpha\Linf(G_1)
=\bigl\{\alpha(\Linf(G_1))\cup\Linf(\hh{G}_2)'\tens\I)\bigr\}''
\end{equation}
The coaction $\beta$ enters in the definition of the comultiplication $\Delta_\GG$. It turns out that the embedding $\alpha\colon\Linf(G_1)\to\Linf(\GG)$ is compatible with the comultiplications:
\begin{equation}\label{inter}
\Delta_{\GG}\circ\alpha=(\alpha\tens\alpha)\circ\Delta_{G_1}.
\end{equation}

\begin{proposition}\label{Xdef}
Define $\XX$ so that $\Linf(\XX)=\alpha\bigl(\Linf(G_1)\bigr)$. Then $\XX$ is a strongly embeddable quantum homogeneous space for $\GG$ with the \cst-version $\alpha\bigl(\C_0(G_1)\bigr)$.
\end{proposition}

\begin{proof}
Using Proposition \ref{cvhProp} with $\gamma=\alpha$ we see that $\XX$ indeed is an embeddable quantum homogeneous space for $\GG$. In order to prove that $\XX$ is strongly embeddable let us consider the canonical unitary implementation of $\alpha$ (\cite{unitImpl}, cf.~\cite[Proof of Proposition 3.6]{BSV}):
\[
U\in\M\bigl(\C_0(G_2)\tens\cK(\Ltwo(G_1))\bigr)\subset\B\bigl(\Ltwo(G_2)\tens\Ltwo(G_1)\bigr),\qquad\alpha(x)=U(\I\tens{x})U^*.
\]
More precisely, the unitary operator $U$ is given by the restriction to $G_2$ of the quasi-regular representation of $G$ on $\Ltwo(G/G_2)$ i.e.~$U\colon{G_2}\ni{g}\mapsto{U_g}\in\B\bigl(\Ltwo(G/G_2)\bigr)$. In order to interpret $U$ as an element $\M\bigl(\C_0(G_2)\tens\cK(\Ltwo(G_1))\bigr)$ we use the identification $\Ltwo(G/G_2)\cong\Ltwo(G_1)$. Using further identifications $\Ltwo(\GG)\cong\Ltwo(G)\cong\Ltwo(G_2)\tens\Ltwo(G_1)$ we note that for any $x\in\C_0(G_1)$ we have
\[
\begin{split}
\Delta_{\GG}\bigl(\alpha(x)\bigr)&=(\alpha\otimes\alpha)\bigl(W^{G_1}(x\tens\I)(W^{G_1})^*\bigr)\\
&=U_{12}\bigl((\id\tens\alpha)(W^{G_1})\bigr)_{234}U_{12}^*\bigl(\alpha(x)\tens\I\bigr)U_{12}
\bigl((\id\tens\alpha)(W^{G_1})\bigr)^*_{234}U_{12}^*\\
&=T\bigl(\alpha(x)\tens\I\bigr)T^*,
\end{split}
\]
where $W^{G_1}$ is the Kac-Takesaki operator of $G_1$ and
\[
T=U_{12}\bigl((\id\tens\alpha)(W^{G_1})\bigr)_{234}U_{12}^*\in\M\bigl(\cK(\Ltwo(\GG))\tens\alpha(\C_0(G_1))\bigr).
\]
Summarizing we conclude that $\XX$ is a strongly embeddable quantum homogeneous space with \cst-version $\C_0(\XX)=\alpha\bigl(\C_0(G_1)\bigr)$.
\end{proof}

The dual of $\GG$ is obtained by exchanging the roles of $G_1$ and $G_2$ so that
\[
\Linf(\hh{\GG})=\bigl\{\beta(\Linf(G_2))\cup\I\tens\Linf(\hh{G}_1))\bigr\}''.
\]
In particular we see that $\hh{G}_2$ is a closed quantum subgroup of $\GG$ in the sense of Vaes via the morphism $\beta\colon\Linf(G_2)\to\Linf(\hh{\GG})$. In what follows we will show that $\XX=\GG/\hh{G}_2$ where $\XX$ was described in Proposition \ref{Xdef}.

\begin{proposition}
The strongly embeddable quantum homogeneous space $\XX$ defined in Proposition \ref{Xdef} is of quotient type: $\Linf(\XX)=\Linf(\GG/\hat{G}_2)$.
\end{proposition}

\begin{proof}
Using the reasoning of the proof of Lemma \ref{charQT} we see that $x\in\Linf(\GG/\hat{G}_2)$ if and only if $x\in\Linf(\GG)$ and $x$ commutes with $\beta\bigl(\Linf(G_2)\bigr)$ (the commutation holds in $\B\bigl(\Ltwo(G_2)\tens\Ltwo(G_1)\bigr)$). Since $x\in{G_2}\ltimes_\alpha\Linf(G_1)$ and the second leg of $G_2\ltimes_\alpha\Linf(G_1)$ commutes with $\Linf(G_1)$ (see \eqref{bcp}), we conclude that $x$ also commutes with $\I\vtens\Linf(G_1)$. From the fact that $\Linf(G_2)\vtens\Linf(G_1)$ is generated by $\beta\bigl(\Linf(G_2)\bigr)$ and $\I\vtens\Linf(G_1)$ it follows that $x$ belongs to the commutant of $\Linf(G_2)\vtens\Linf(G_1)$. Since the von Neumann algebra $\Linf(G_2)\vtens\Linf(G_1)$ is a maximal commutative subalgebra of $\B\bigl(\Ltwo(G_2)\tens\Ltwo(G_1)\bigr)$ w get $x\in\Linf(G_2)\vtens\Linf(G_1)$. Proceeding as in the proof of Theorem \ref{ddddT} we see that $x$ is invariant under the dual action on $G_2\ltimes_\alpha\Linf(G_1)$. This shows that $\Linf(\GG/\hat{G}_2)\subset\alpha\bigl(\Linf(G_1)\bigr)$ and since the opposite inclusion is clear, we get the equality $\Linf(\XX)=\alpha\bigl(\Linf(G_1)\bigr)=\Linf(\GG/\hat{G}_2)$.
\end{proof}

\subsection{Quantum homogeneous space of the quantum ``$az+b$'' group}\label{azbsubs}

In this subsection let $\GG$ be the quantum ``$az+b$'' group (see \cite{azb,nazb} and \cite[Section 1]{PuszSol1} for a general approach) for an admissible value of the deformation parameter (cf.~\cite[Section 1]{PuszSol1}). In \cite[Section 3]{Psol} and \cite{PuszSol0} a construction of a quantum space $\YY$ endowed with a continuous action of $\GG$ was carried out. The space $\YY$ turned out to be a classical space (the \cst-algebra $\C_0(\YY)$ is commutative) and $\YY$ was called a ``homogeneous space'' for $\GG$. We will shed more light on this example.

Let $\delta_{\YY}\in\Mor\bigl(\C_0(\YY),\C_0(\GG)\tens\C_0(\YY)\bigr)$ be the action of $\GG$ on $\YY$. By the results of \cite{azb,nazb,Psol} we have
\begin{equation}\label{eFlabel}
\delta_\YY(x)=U(x\tens\I)U^*,
\end{equation}
where $U\in\M\bigl(\cK(\Ltwo(\GG))\tens\C_0(\YY)\bigr)$ is defined using the \emph{quantum exponential function} (\cite[Section 1]{PuszSol1}) and its form depends on the value of the deformation parameter (see also \cite[Proof of Theorem 3.1]{Psol}).

As $\C_0(\YY)$ comes equipped with an embedding into $\B\bigl(\Ltwo(\GG)\bigr)$ we can immediately put ``measurable structure'' on $\YY$ by setting $\Linf(\YY)=\C_0(\YY)''$. It is fairly clear form \cite[Section 3]{Psol} (cf.~also \cite{PuszSol0}) that there is a classical (in fact abelian) locally compact group $\Gamma$ such that $\Gamma$ is a closed quantum subgroup of $\GG$. Using standard facts about crossed products by abelian groups one easily sees that $\Linf(\YY)$ is precisely the set of those $x\in\Linf(\GG)$ which commute with the image of the associated inclusion $\Linf(\hh{\Gamma})\hookrightarrow\Linf(\GG)$. This means that $\YY=\GG/\Gamma$. In particular we have:

\begin{proposition}
$\YY$ is a strongly embeddable quantum homogeneous space for $\GG$ of quotient type.
\end{proposition}

The classical ``$az+b$'' group is a semidirect product of $\CC$ by the action of $\CC^\times=\CC\setminus\{0\}$. It follows that if $G$ denotes the classical ``$az+b$'' group then the bundle $G\to{G/\CC^\times}$ is trivial. In particular the homogeneous space $G/\CC^\times=\CC$ is projectively embeddable. Interestingly this is not the case on the quantum level.

\begin{proposition}\label{nonProj}
Let $\GG$ be the quantum ``$az+b$'' group for non-trivial deformation parameter (so that $\GG$ is not the classical ``$az+b$'' group). Then the quantum homogeneous space $\YY$ for $\GG$ is not projectively embeddable.
\end{proposition}

\begin{proof}
Assume that we have a morphism $\pi\in\Mor\bigl(\C_0(\GG),\C_0(\GG)\bigr)$ satisfying conditions of Definition \ref{projem}. In particular $\pi$ maps $\C_0(\GG)$ into a commutative \cst-algebra $\C_0(\YY)$. It is known that $\C_0(\GG)$ is a crossed product of $\C_0(\YY)$ by an action of the group $\Gamma$ mentioned above (see \cite{azb,nazb,PuSoFunct,PuszSol0,PuszSol1} for various descriptions of this fact). Since the original action is inner in the crossed product, the morphism $\pi$ restricted to $\C_0(\YY)$ would have to map into the fixed point subalgebra (of the action of $\Gamma$) of $\M\bigl(\C_0(\YY)\bigr)$. But this subalgebra is equal to $\CC\I$. On the other hand $\pi$ must be the identity on $\C_0(\YY)$. Hence we arrive at a contradiction.
\end{proof}

As mentioned in the proof of proposition \ref{nonProj} the \cst-algebra $\C_0(\GG)$ is a crossed product of $\C_0(\YY)$ by an action of $\Gamma$. This situation is explicitly considered as a ``trivial quantum principal bundle'' by some authors (see e.g.~\cite[Definition 2.2]{HaZi}). This underlines our point that the principal bundle corresponding to the action of $\GG$ on $\YY$ should be considered as a non-commutative ``trivial bundle'' and strengthens the case that the notion of strong embeddability captures this phenomenon.

\section*{Acknowledgements}
The authors wish to thank Matthew Daws, Adam Skalski, and Stanis{\l}aw Woronowicz for stimulating discussions on the subject of this work.

\end{document}